%% file: main.tex
\documentclass{amsart}

\usepackage[foot]{amsaddr}

\usepackage{subfig}
\usepackage{graphicx}
\usepackage{amssymb,amsfonts,amsrefs}
\usepackage{hyperref}
\usepackage{mathcommands}
\usepackage{tikz-cd}
\usepackage{enumerate}
\usepackage{array}
\usepackage{bm}
\usepackage{pgfplots}
\pgfplotsset{compat=1.18}
\usepackage[toc,page]{appendix}
\usepackage{multirow}

\newcommand{\medmatrix}[2][.8]{%
  \scalebox{#1}{%
    \renewcommand{\arraystretch}{.8}%
    $\begin{pmatrix}#2\end{pmatrix}$%
  }
}

\usepackage{makecell}
\newcolumntype{K}[1]{>{\centering\arraybackslash}m{#1}}

\usepackage{changes}

\newcommand{\ldVertex}[4]{\node (#1) at (#2) [inner sep=2pt] {#3 {\tiny [#4]}};}
\newcommand{\ldEdge}[4]{\draw[->] (#1) edge[#4] node[fill=white,circle, inner sep =1pt,scale=.7] {#3} (#2);}

\usepackage[OT2,T1]{fontenc}
\DeclareSymbolFont{cyrletters}{OT2}{wncyr}{m}{n}
\DeclareMathSymbol{\Sha}{\mathalpha}{cyrletters}{"58}
\DeclareMathSymbol{\B}{\mathalpha}{cyrletters}{"42}

\title[]{Bockstein Spectral Sequences and Applications to the Tame Fontaine Mazur Conjecture}

\author{Julian Feuerpfeil}

\address{Dipartimento di Matematica e Applicazioni, Università di Milano-Bicocca, 20125 Milano, Italy}

\address{FEMTO-ST, Université Marie et Louis Pasteur,
25030 Besançon, France}

\email{\href{mailto:j.feuerpfeil@campus.unimib.it}{j.feuerpfeil@campus.unimib.it}}

\begin{document}
% ABSTRACT
\begin{abstract} 
For a number field $K$, a prime $p$ and a finite set of tame places $S$ we consider the groups $G_{K,S}$ --- the Galois groups of the maximal pro-$p$ extension of $K$ unramified outside $S$. The tame Fontaine--Mazur Conjecture predicts that these groups have no nontrivial uniformly powerful pro-$p$ quotients. 

In this paper we develop a new approach to this problem using Bockstein spectral sequences and Lie-theoretic tools. This allows us to extend and refine an earlier method due to J.~Labute, who was only able to consider the case where $p^2\nmid N(\mathfrak{q})-1$ for each $\mathfrak{q}\in S$.

Based on this we develop a method to verify the uniform Fontaine--Mazur property for many $G_{K,S}$ with $|S|=3$ arbitrary. Under mild conditions on $K$ we show that for infinitely many triples $S$, the groups $G_{K,S}$ have no nontrivial uniform quotients. We also exhibit a large class of these groups, which are infinite. Finally, we present numerical evidence indicating that the criteria developed here detect the uniform Fontaine--Mazur property with very high probability for $|S|=3$.

\noindent \textbf{Keywords.} uniform Fontaine--Mazur Conjecture, Linking diagrams, Bockstein--Spectral sequences, uniform pro-$p$ groups

\noindent \textbf{MSC2020} Primary: 11F80 Secondary: 11R34 20J06, 17B56

\noindent \textbf{ORCID:} \ \href{https://orcid.org/0009-0000-0148-3348}{0009-0000-0148-3348}

%\noindent \textbf{e-mail:} \ \ \ \href{mailto:j.feuerpfeil@campus.unimib.it}{j.feuerpfeil@campus.unimib.it}

\end{abstract}

\maketitle

%\tableofcontents

\input{contents}

\section*{Acknowledgments}
I would like to thank, first and foremost, my supervisors Thomas Weigel and Christian Maire for posing this research question to me, for guiding and supporting me throughout this project, and their careful proof-reading of this manuscript. I thank Simone Blumer, Oussama Hamza, Holger Kammeyer, Benjamin Klopsch, and Donghyeok Lim for stimulating discussions, helpful comments and their interest in this research. I am grateful to Andrei Jaikin-Zapirain and John Labute for their interest in this work.

I am a member of Gruppo Nazionale per le Strutture Algebriche, Geometriche e le loro Applicazioni (GNSAGA), which is part of the Istituto Nazionale di Alta Matematica (INdAM). Furthermore, I participate in the \lq\lq INdAM - GNSAGA Project\rq\rq\ CUP E53C25002010001.

\appendix
\input{appendix}

\bibliography{references}
\newpage

\end{document}

%% file: contents.tex
\section*{Introduction}
Let $K$ be a number field. The $p$-adic étale cohomology of smooth projective varieties over $K$ gives rise to geometric representations of the absolute Galois group $\Gamma_K$ of $K$. These are characterized by ramification at a finite set of places together with local $p$-adic Hodge theoretic conditions. In \cite{FontaineMazur1995} J-M.~Fontaine and B.~Mazur posed their celebrated \emph{Fontaine--Mazur Conjecture}, asserting that every geometric representation arises in this way. Note that there exist many representations of $G_K$ which are not geometric (cf.,  \cites{Ramakrishna2008,HajirLarsenMaireRamakrishna2025}).

The condition that geometric representations are only ramified at finitely many places implies that the associated homomorphism $\rho:G_K\to \GL_n(\ZZ_p)$ factors through $\Gamma_{K,S}$ --- the Galois group of the maximal extension $K_S/K$ unramified outside $S$ --- for a suitable finite set of places $S$. If no place in $S$ lies above $p$, then $\rho:\Gamma_{K,S}\to \GL_n(\ZZ_p)$ is geometric and Conjecture 5a of \cite{FontaineMazur1995} asserts that the image of $\rho$ is finite, being equivalent to the condition that $\Gamma_{K,S}$ does not admit any infinite $p$-adic analytic quotients. It was only recently shown, by \cite{KisinWortmann2003} and \cite{Moonen2019}, that this statement follows from the (general) Fontaine--Mazur Conjecture together with Tate's conjecture on algebraic cycles.

We denote the maximal pro-$p$ quotient of $\Gamma_{K,S}$ by $G_{K,S}$. Since every $p$-adic analytic group contains an open uniform group (see \cite{DDMS1999}) it follows that \cite{FontaineMazur1995}*{Conj. 5a} is equivalent to the following.
\begin{conj*}[uniform Fontaine--Mazur]
    For each number field $K$ and each finite set of tame places $S$, the group $G_{K,S}$ does not admit a nontrivial uniform quotient.
\end{conj*}

Motivated by this conjecture, we say that a pro-$p$ group $G$ has the \emph{uniform Fontaine--Mazur} property (or short \emph{uFM} property) if $G$ admits no nontrivial uniform quotients.
\begin{rem*}
    Given a group $G$, the property of not admitting infinite uniform quotients is weaker than the property of not admitting infinite $p$-adic analytic quotients.  For example the group $\SL_1^1(\Delta_p)$, where $\Delta_p$ a maximal order in a non-split quaternion algebra over $\QQ_p$, is just-infinite and not uniform, but $p$-adic analytic (cf. \cite{GonzaleSanchezKlopsch2009}). 
\end{rem*}
In \cite{Labute2014} J.~Labute posed a Lie-theoretic approach to the tame Fontaine--Mazur conjecture, by constructing a map 
\begin{align*}
    \ell:\Hom_{\rm cont}(G_{\QQ,S},\GL_n^1(\ZZ_p))\to \Hom(\mathfrak{l}_S,\mathfrak{gl}_n(\FF_p))
\end{align*}
where $\mathfrak{l}_S$ is an $\FF_p$-Lie algebra mimicking the presentation of $G_{\QQ,S}$ determined by H.~Koch (see \cites{Koch1966,Koch2002}). This map has the property that $\ell(\rho)=0$ if and only if $\rho$ is trivial. Hence if $\mathfrak{l}_S$ does not admit nontrivial (finite-dimensional) representations, then the same is true for representations $G_{\QQ,S}\to \GL_n^1(\ZZ_p)$.  

The definition of $\mathfrak{l}_S$ and the fact that only $\FF_p$-representations are considered, inherently limits this approach to sets of primes $S=\{q_1,...,q_d\}$ for which $q_i\not\equiv 1\pmod {p^2}$ for all $i=1,...,d$. {\color{black} In this paper we develop an approach allowing us to consider primes $q_i$, for which $v_p(q_i-1)$ is large, by studying $\ZZ_p$-Lie algebras. Our strategy is the following:} 

\subsection*{Strategy}
{\color{black}
A central construction in our approach is the so-called \emph{Bockstein spectral sequence} associated to (pro-$p$) groups $G$ and $\ZZ_p$-Lie algebras $\Lie$. We denote them by $(\BS_r^*(G),\beta^*_r)$ and $(\BS_r^*(\Lie),\mathfrak{b}_r^*)$. Unlike spectral sequences arising from filtrations, Bockstein spectral sequences are naturally one-graded and are derived from the long cohomology sequence associated to the coefficient sequence $0\to \ZZ_p\to \ZZ_p\to \FF_p\to 0$. We refer to Section~\ref{sec:Bockstein spectral sequences} for the construction and further explanations. In the first page they agree with $\rmH^*(G,\FF_p)$ resp. $\rmH^*(\Lie,\FF_p)$ and the differentials are given by the classical Bockstein homomorphisms. These are the connecting homomorphisms $\rmH^i(G,\FF_p)\to \rmH^{i+1}(G,\FF_p)$ and $\rmH^i(\Lie,\FF_p)\to \rmH^{i+1}(\Lie,\FF_p)$ arising from the exact sequence.
\begin{align*}
    0\to \FF_p\to \ZZ/p^2\to \FF_p\to 0.
\end{align*}
The following theorem is the key technical ingredient of the paper. It identifies the Bockstein spectral sequence of a uniform pro-$p$ group with the one of its associated $\ZZ_p$-Lie algebra. The main ingredients were already shown by A.~Huber, G.~Kings, and N.~Naumann in \cite{HuberKingsNaumann2011}.
\begin{mainthm}[Theorem~\ref{thm:uniform Bocksteins agree}]
    Let $U$ be a uniform pro-$p$ group and $\Lie(U)$ be its associated $\ZZ_p$-Lie algebra. Then the Bockstein spectral sequences $(\BS^*_r(U),\beta_r^*)$ and $(\BS^*_r(\Lie(U)),\mathfrak{b}_r^*)$ are naturally isomorphic.
\end{mainthm}
The second page of the Bockstein spectral sequence associated to a uniform pro-$p$ group has been computed by W.~Browder and J.~Pakianathan in \cite{BrowderPakianathan2000}. This result has been used by T.~Weigel in \cite{Weigel2000} to study the rigidity of certain uniform pro-$p$ groups and $\ZZ_p$-Lie algebras. In \cite{BrotoLevi1997} C.~Broto and R.~Levi posed the question whether a finite $p$-group is uniquely determined by its Bockstein spectral sequence. For many groups the answer is positive \cite{DiazRuizViruel2013} and no counterexample is known so far.
}

Using the functoriality of the Bockstein spectral sequences we translate the relations in $G_{K,S}$ into congruence relations inside the $\ZZ_p$-Lie algebra constructed by M.~Lazard \cites{Lazard1965,DDMS1999} of a (potential) uniform quotient. While the first Bockstein differential reproduces the classical congruences considered by J.~Labute, higher differentials yield congruences modulo higher powers of $p$. These stronger congruences frequently force the Lie algebra to be trivial, even when the relations modulo $p$ are not sufficient to conclude the triviality. This allows us to exclude the existence of nontrivial uniform quotients of several pro-$p$ groups, even if $G^{\rm ab}$ is not elementary. 
\begin{exmp*}[Example \ref{exmp:four generators five relations}]
    Let $n$ be a positive integer and $G$ be a pro-$p$ group generated by $x_1,x_2,x_3,x_4$ such that the elements
    \begin{align*}
        x_1^{p^n}[x_1,x_2],\quad x_2^{p^n}[x_2,x_3],\quad x_3^{p^n}[x_3,x_4],\quad x_4^{p^n}[x_4,x_1],\quad [x_1,x_3][x_2,x_4]
    \end{align*}
    are contained in $(G')^p[G',G]$ then $G$ has no nontrivial uniform quotients.
\end{exmp*}
For $n=1$ this was already known to J.~Labute even without the relation $[x_1,x_3][x_2,x_4]$. A key feature of our approach is that $n$ can be any positive integer.

We develop a method summarized in Table~\ref{tab:rk3 Algorithm}, composed of multiple smaller results to verify the uFM property for pro-$p$ groups $G_{K,S}$ for $|S|=3$ and $h_p(K)=1$. In conjunction with Theorem~\ref{thm:Linking diagram approximation by arithmetic} we show the following result also for number fields under mild assumptions.
\begin{mainthm}[Theorem~\ref{thm:realization of uFM groups} for $K=\QQ$]
    Let $p$ be an odd prime, $a_1,a_2,a_3\in \ZZ\setminus \{0\}$ and $k\geq 1$ there exist infinitely many sets $S=\{q_1,q_2,q_3\}$ of tame primes with $N(q_i)-1\equiv pa_i\pmod {p^k}$ such that $G_{\QQ,S}$ has the uFM property.
\end{mainthm}
Section~\ref{ssec:Statistics on uFM property} presents some numerical experiments that indicate that these conditions occur with very high frequency and, heuristically, with probability approach $1$ as $p$ increases.

Since in the case $|S|=3$ the Golod--\v{S}afarevi\v{c} inequality is not directly applicable, it is unclear a priori whether the groups $G_{K,S}$ are infinite and thus would satisfy the uFM property for trivial reasons. In fact it follows from an argument by I.V.~ Andozhskij and V.M.~Tsvetkov \cite{AndozhskijTsvetkov1975} that a large portion of the groups $G_{\QQ,S}$ with $|S|=3$ is finite (see for example \cite{Luo2024}*{Thm. 7.13}). See Remark~\ref{rem:why our method is need for infinity} for a more detailed discussion. Using the method from Section~\ref{ssec:FAb groups with three generators} we construct triples of primes $S$, such that $G_{K,S}$ satisfies the uFM property and is infinite. To that end we use a strategy similar to the one by F.~Hajir and C.~Maire in \cite{HajirMaire2002} to construct infinite unramified extensions. 
\begin{mainthm}[Theorem~\ref{thm:Infinite and uFM GS}]
    For a number field $K$ and an odd prime $p$ unramified in $K/\QQ$ with $h_p(K)=1$ there are infinitely many triples of primes $S=\{v_1,v_2,v_3\}$ such that $G_{K,S}$ has the uFM property and $G_{K,S}$ is infinite.
\end{mainthm}
Even though the method we develop in Section~\ref{ssec:FAb groups with three generators} is not directly applicable if $|S|>3$, we show in Theorem~\ref{thm:additonal d-2 splitting} and Corollary~\ref{cor:One Splitting for S=4} that by introducing a well-chosen small set of additional places $T$ one can conclude the uFM property for $G_{K,S}^T$ instead of $G_{K,S}$.

\subsection*{Structure of the paper}
Section~\ref{sec:GS and Koch groups} recalls the arithmetic background and abstract Koch presentations. Sections~\ref{sec:Bockstein spectral sequences} and \ref{sec:Uniform quotients} introduce the Bockstein spectral sequence techniques and derive congruence criteria for uniform quotients. Section~\ref{sec:Koch type groups with three generators} applies these methods to groups with three generators, proving the main arithmetic results. Next in Section~\ref{sec:Additional splitting} we study the uFM property, with additional splitting conditions. Section~\ref{sec:numerical experiments} discusses statistics for linking numbers and $G_{K,S}$ with three generators, that satisfy the uFM property.

\section{The groups \texorpdfstring{$G_{K,S}$}{GKS} and Koch-type groups}
\label{sec:GS and Koch groups}
Let $K$ be a number field, $p$ a prime and $S$ a finite set of \emph{tame} places of $K$, i.e., $N(\mathfrak{q})\equiv 1\pmod p$ for all $\mathfrak{q}\in S$. Denote by $h_p(K)$ the maximal $p$-power dividing the class number of $K$. We define $K_S(p)/K$ to be the maximal pro-$p$ extension of $K$ unramified outside $S$ and by $G_{K,S}$ its Galois group. It coincides with $\pi_1^{\rm \acute{e}t}(\Spec(\mathcal{O}_{K,S}))(p)$. The main purpose of this section is to recall some basic facts about $G_{K,S}$ and to abstract these properties to the realm of pro-$p$ groups.
\subsection{The Galois Cohomology of \texorpdfstring{$G_{K,S}$}{GKS}}
\label{ssec:The Galois Cohomology of GS}
All the following results are classical and can be found in~\cites{Koch2002,NSW2000}. Fix a number field $K$ and an odd prime $p$. {\color{black}The situation is more technical for $p=2$ and since we will only be interested in $p$ odd in what follows, we only consider this case.} For two finite sets of places $S$, $T$ (not necessarily disjoint), one defines
\begin{align*}
    V^T_S(K)\coloneq \{a\in K^\times: a_v\in K_v^{\times p}\text{ for }v\in S\text{ and }p\mid w(a)\text{ for all }w\not\in T\}/K^{\times p}.
\end{align*}
The Pontryagin dual of $V_S^T(K)$ will be denoted by $\B_S^T(K)$. For $S\subseteq S'$ the natural map $\B_S^T\to \B_{S'}^T$ is surjective and hence from the exact sequence
\begin{align*}
    0\to \mathcal{O}_{K,T}^\times/\mathcal{O}_{K,T}^{\times p}\to V_{\emptyset}^T\to \cl^T(K)[p]\to 0
\end{align*}
it follows that $\B_S^T(K)$ is finite for all finite $S$ and $T$. If either $S$ or $T$ are empty we suppress it in the notation. 

Lemma~10.7.4 of \cite{NSW2000} yields the following exact sequence
\begin{align}
\label{eq:generators of G_KS sequence}
    0\to \rmH^1(G_{K,\emptyset},\FF_p)\to \rmH^1(G_{K,S},\FF_p)\to \bigoplus_{v\in S}\rmH^1(\mathcal{T}_v,\FF_p)^{\mathcal{G}_v}\to \B_\emptyset(K)\to \B_S(K)\to 0
\end{align}
where $\mathcal{T}_v$ is the inertia group of $\mathcal{G}_v=\Gal(\overline{K}_v/K_v)$. Note that $\rmH^1(G_{K,\emptyset},\FF_p)$ is by class field theory isomorphic to $\Hom(\cl(K),\FF_p)$ and therefore trivial if $h_p(K)=1$. In this situation, the exact sequence implies, that $G_{K,S}$ is generated by the inertia groups of the $v\in S$.

Denote by $\Sha^2(G_{K,S})$ the kernel of the map $\rmH^2(G_{K,S},\FF_p)\to \prod_{v\in S}\rmH^2(\mathcal{G}_v,\FF_p)$. Then the groundbreaking result due to H.~Koch is, that $\Sha^2(G_{K,S})$ embeds into $\B_S(K)$. Therefore all relations in $G_{K,S}$ are coming from local ones provided $\B_S(K)=0$.

From that one deduces the following corollary, which can be found in a more general form as \cite{NSW2000}*{Cor. 10.7.7}. We present it in that way, since we are only interested in this particular case.
\begin{cor}
\label{cor:generator and relation rank of GS}
    Let $K$ be a number field {\color{black} with signature $(r_1,r_2)$ and $r\coloneq r_1+r_2$}, $p$ an odd prime such that $\zeta_p\not \in K$ and $S$ be a set of tame primes, then 
    \begin{align*}
        d(G_{K,S})=1+|S|+\dim_{\FF_p}\B_S(K)-r\quad \text{and}\quad 
        r(G_{K,S})\leq |S|+\dim_{\FF_p}\B_S(K).
    \end{align*}
\end{cor}

\subsection{Abstract Koch-Type presentations}
\label{ssec:Abstract Koch Presentation}
In favorable situations, the groups $G_{K,S}$ admit quite explicit presentations. The central result is due to H.~Koch (see \cites{Koch1966,Koch2002}). Influenced by an analogy between the groups $G_{K,S}$ and fundamental groups of knot complements, J.~Labute introduced in \cite{Labute2006} so-called \emph{linking diagrams} to present combinatorial data associated to a tame set of primes $S$ encoding crucial information about the presentation of the group $G_{K,S}$.

We refine his notion of linking diagram since we need more combinatorial data for our approach.
\begin{defn}
\label{defn:Linking diagram}
    Fix a prime $p$. A \emph{linking diagram} $\Lambda$ is a triple $(\Gamma,a,\ell)$, where $\Gamma=(V,E)$ is a directed graph, $a=(a_v)_{v\in V}$ a sequence of non-zero elements in $\ZZ_p$ and $\ell:V\times V\to \FF_p$ such that $\ell(v,w)=0$ iff $(v,w)\not\in E$ (in particular if $v=w$).

    Two linking diagrams $\Lambda=(\Gamma,a,\ell)$ and $\Lambda'=(\Gamma',a',\ell')$ are considered to be equivalent if there is an isomorphism $\varphi:\Gamma\to \Gamma'$ and $c_v,d_v\in \ZZ_p^\times$ such that $c_va_v=a'_{\varphi(v)}$ and 
    \begin{align*}
        c_vd_w\ell(v,w)=\ell'(\varphi(v),\varphi(w))\qquad \text{for all }w\in V.
    \end{align*}
\end{defn}
We represent linking diagrams $\Lambda=(\Gamma,a,\ell)$ as follows, illustrated for a fully connected graph with three vertices $u,v,w$.
\begin{center}
    \begin{tikzpicture}
        \ldVertex{v}{90:1.5}{$v$}{$a_v$}
        \ldVertex{w}{200:1.5}{$w$}{$a_w$}
        \ldVertex{u}{340:1.5}{$u$}{$a_u$}
        \ldEdge{v}{w}{$\ell(v,w)$}{bend right=1.5cm}
        \ldEdge{w}{v}{$\ell(w,v)$}{bend right=.5cm}
        \ldEdge{u}{w}{$\ell(u,w)$}{bend right=.5cm}
        \ldEdge{w}{u}{$\ell(w,u)$}{bend right=1.5cm}
        \ldEdge{v}{u}{$\ell(v,u)$}{bend right=.5cm}
        \ldEdge{u}{v}{$\ell(u,v)$}{bend right=1.5cm}
    \end{tikzpicture}
\end{center}
\begin{defn}
\label{defn:groups represented by linking diagrams}
    Let $\Lambda=(\Gamma,a,\ell)$ be a linking diagram and $F$ the free pro-$p$ group  on generators $x_v$ for $v\in V(\Gamma)$. 
    
    A pro-$p$ group $G$ is said to be \emph{weakly presented by $\Lambda$} if there exists an epimorphism $\pi:F\twoheadrightarrow G$, such that $\pi(x_v)$ is a minimal set of generators of $G$ and for each $v\in V(\Gamma)$ there exist $\rho_v\in \ker(\pi)$ such that 
    \begin{align*}
        \rho_v\equiv x_v^{pa_v}\prod_{w\in V(\Gamma)}[x_v,x_w]^{\ell(v,w)}\pmod{P_3^*(F)}.
    \end{align*}
    We say that $G$ is \emph{presented by $\Lambda$}, if $\ker(\pi)$ is equal to the closed normal subgroup generated by the $\rho_v$, i.e., $G\cong \langle x_v\mid \rho_v\rangle $.
    
    A pro-$p$ group is said to have a \emph{(weak) Koch-type presentation} if it is (weakly) presented by some linking diagram $\Lambda$.
\end{defn}
\begin{rem}
    Note that there are many pro-$p$ groups presented by the same linking diagram. Furthermore, if $\Lambda$ and $\Lambda'$ are equivalent and $G$ is (weakly) presented by $\Lambda$, then it is also (weakly) presented by $\Lambda'$.
\end{rem}
\begin{con}
    Let $p$ be a prime number and $K$ a number field such that $h_p(K)=1$. Given a set of tame places $S$ such that $\B_S=\B_\emptyset$, then for each $v\in S$, there exists a cyclic extension $L(v)/K$ of degree $p$ unramified exactly at $v$ such that $K_S^{\rm el,ab}$ is the compositum of all $L(v)$ (this follows either from \cite{Koch2002}*{Thm. 11.7} or \cite{NSW2000}*{Lem. 10.7.4 (i)}). 

    For each $v\in S$ fix a generator $\tau_v$ of $\Gal(L(v)/K)$ and for $v\neq w\in S$ we define $\ell(v,w)\in \FF_p$ by
    \begin{align*}
        (v,L(w)/K)=\tau_v^{-\ell(v,w)}\in \Gal(L(w)/K).
    \end{align*}
    Let $\Gamma_S=(S,E_S)$ with $E_S\coloneq \{(v,w)\mid \ell(v,w)\neq 0\}$, and set $(a_S)_v\coloneq(N(\mathfrak{q}_v)-1)/p$. Then $\Lambda_{K,S}\coloneq(\Gamma_S,a_S,\ell)$ is the linking diagram associated to $K$ and $S$.

    Note that a different choice of generators $\tau_v'$ yields an equivalent linking diagram.
\end{con}
\begin{exmp}
    For $K=\QQ$ and $p=3$ consider $S=\{13,19,37,109\}$. With respect to the primitive roots $2$, $10$, $19$ and $51$ we get the following linking diagram:
    \begin{center}
        \begin{tikzpicture}[]
            \node at (-1.5,1) {$\Lambda_{\QQ,S}=$};
            \ldVertex{13}{0,2}{13}{4};
            \ldVertex{19}{0,0}{19}{6};
            \ldVertex{37}{2,0}{37}{12};
            \ldVertex{109}{2,2}{109}{36};
            
            \ldEdge{109}{19}{1}{%bend left= 1cm
            };
            \ldEdge{13}{19}{2}{bend right=1.75cm};
            \ldEdge{13}{37}{2}{bend right=1cm};
            \ldEdge{13}{109}{2}{bend left=.5cm};
            \ldEdge{19}{13}{1}{bend left=.5cm};
            \ldEdge{19}{37}{2}{bend right=.5cm};
            \ldEdge{37}{13}{2}{bend right=1cm};
            \ldEdge{37}{109}{1}{bend right=1.75cm};
            
            \ldEdge{109}{37}{1}{bend left= .5cm};
            
        \end{tikzpicture}
    \end{center}
\end{exmp}
The following theorem relates the linking diagrams we defined above to the groups $G_{K,S}$ and their representations.
\begin{thm}[Koch \cite{Koch2002}]
    Let $p$ be an odd prime number, $K$ a number field with $h_p(K)=1$, and $S$ a set of tame places such that $\B_S=\B_\emptyset$. Then $G_{K,S}$ is weakly presented by $\Lambda_{K,S}$. Furthermore, if $\B_\emptyset=0$, i.e., if $K=\QQ$ or $K$ is imaginary quadratic, then $G_{K,S}$ is presented by $\Lambda_{K,S}$.
\end{thm}

\subsection{Realization of linking numbers}
\label{ssec:Realization of linking numbers}
In \cite{Labute2006} J.~Labute asked whether each possible choice of linking numbers is realizable by a tame set of places. This question was recently answered by the author together with O.~Hamza and D.~Lim in \cite{FeuerpfeilHamzaLim2026}. We refine the statement here in terms of linking diagrams from Definition~\ref{defn:Linking diagram}, which carry the sequence $a$ as additional piece of data.
\begin{thm}
\label{thm:Linking diagram approximation by arithmetic}
    Let $p$ be an odd prime, $\Lambda$ be linking diagram and $K$ a number field, not containing $\zeta_p$ with $h_p(K)=1$ and $p$ unramified in $K/\QQ$. Then there exists a set of tame places $S$ of $K$ such that $\Lambda$ is equivalent to $\Lambda_{K,S}$. Moreover, such a set of places is given by \v{C}ebotarev classes and therefore they have a positive density.
\end{thm}

The above theorem follows by an inductive argument from the following rather technical Lemma. Its proof is based on a refinement of the Gras--Munnier theorem (see \cite{GrasMunnier1998}) due to C. Maire and K. Sankara \cite{MaireSankara2025} and follows verbatim the one in \cite{FeuerpfeilHamzaLim2026}*{Lem. 2.13} with minimal modifications. Our assumptions on $p$ and $K$ guarantee by Remark 2.14 of \cite{FeuerpfeilHamzaLim2026} that we can choose the norm of $\fp_{v'}$ arbitrarily.
\begin{lem}
\label{lem:construction of extension with prescribed ramification}
    Let $p$ be an odd prime and $K$ a number field not containing $\zeta_p$ with $h_p(K)=1$ and $p$ unramified in $K/\QQ$. Fix a set $S$ of tame places of $K$. Let an element $\tau\in G_{K,S}^{\rm el,ab}$, $a\in p\ZZ_p$, $k\geq 1$ and elements $\ell_{v}\in \FF_p$ for $v\in S$ be given.  Then there exists a prime $v'$ of $K$ and a cyclic extension $L/K$ of degree $p$ such that
    \begin{enumerate}
        \item \label{it:Norm condition} $N(\fp_{v'})-1\equiv a\pmod {p^k}$;
        \item \label{it:Frobenius condition} $(v',K_S^{\rm el,ab}/K)=\tau$;
        \item \label{it:Ramification condition} $L/K$ is ramified exactly at $v'$;
        \item \label{it:linking condition} for a generator $\sigma$ of $\Gal(L/K)$ we have $(v,L/K)=\sigma^{\ell_{v}}$ for each $v\in S$.
        \item The canonical surjection $\B_{S}\twoheadrightarrow \B_{S\cup \{v'\}}$ is an isomorphism.
    \end{enumerate}
\end{lem}
\begin{rem}
    To achieve \eqref{it:Norm condition} in Lemma~\ref{lem:construction of extension with prescribed ramification} we need that $K\cap \QQ(\zeta_{p^n})=\QQ$. This is the case for example if $p$ is unramified in $K/\QQ$. For example for $K=\QQ(\zeta_{p^n})^+$ the maximal totally real subfield of $\QQ(\zeta_{p^n})$ every tame place $v$ of $K$ satisfies $p^n\mid N(\fp_v)-1$.
\end{rem}

\section{Bockstein spectral sequences}
\label{sec:Bockstein spectral sequences}
Before returning to the study of the groups $G_{K,S}$ we introduce some homological techniques --- in particular Bockstein spectral sequences for pro-$p$ groups and $\ZZ_p$-Lie algebras.

These spectral sequences are defined in terms of exact couples. They were first introduced by W.S.~Massey in \cite{Massey1953}. An \emph{exact couple} in an abelian category $\mathcal{A}$ is a tuple $(A,E,\alpha,f,g)$, where $A$ and $E$ are objects of $\mathcal{A}$ and $\alpha,f,g$ are morphisms as in the following diagram
\begin{equation*}
    \begin{tikzcd}[column sep=small]
        A\arrow[rr,"\alpha"]&&A\arrow[dl,"f"]\\
        &E\arrow[ul,"g"]
    \end{tikzcd}
\end{equation*}
such that $\ker(\alpha)=\im(f)$, $\ker(f)=\im(\alpha)$, and $\ker (g)=\im(\alpha)$. {\color{black} To an exact couple $\mathcal{D}=(A,E,\alpha,f,g)$, one constructs the so-called \emph{derived couple} given as $\mathcal{D}'=(\alpha(A),E',\alpha|_{\alpha(A)},f',g')$ with 
\begin{align*}
    E'\coloneq \tfrac{\ker(f\circ g)}{\im(f\circ g)},\quad f'(\alpha(x))\coloneq [f(x)],\quad \text{and}\quad g'([y])=g(y).
\end{align*}
It is straightforward to verify that this tuple is again an exact couple. By iterating this procedure, one gets a sequence of exact couples $(\mathcal{D}^{(r)})_{r\geq 0}$, which we will call the \emph{spectral sequence associated to} $\mathcal{D}=\mathcal{D}^{(0)}$. There is a natural notion of morphisms of exact couples, which give rise to morphisms of the associated spectral sequences. For more details and further applications of exact couples, we refer to \cite{Massey1953}, \cite{stacks-project}*{\href{https://stacks.math.columbia.edu/tag/011P}{tag 011P}}, and \cite{McCeary2001}*{\S 2.2}.}
\begin{con}
\label{con:Bockstein Sequence of Complex}
    Let $\mathcal{A}$ be an abelian category and $C^*\in {\rm Ch}(\mathcal{A})$ a complex. Let $\varphi:C^*\to C^*$ be a monomorphism with cokernel $D^*$. Setting $A\coloneq \bigoplus_{n}\rmH^n(C^*)$ and $E\coloneq\bigoplus_{n}\rmH^n(D^*)$ the long exact sequence in cohomology induces maps $\varphi^*:A\to A$, $\pi^*:A\to E$ and $\delta^*:E\to A$, where $\varphi^*$ and $\pi^*$ are degree preserving and $\delta^*${\color{black}, being induced by the connecting homomorphism,} increases the degree by $1$. 

    The tuple $(A,E,\varphi^*,\pi^*,\delta^*)$ is an exact couple and the resulting spectral sequence $(\BS_r(C^*,\varphi),d_r)_{r\geq 1}$ is called the \emph{Bockstein spectral sequence} associated to $C^*$ and $\varphi$.
\end{con}
\subsection{The Bockstein spectral sequence for pro-\texorpdfstring{$p$}{p} groups}
\label{ssec:Bockstein spectral sequence pro-p groups}
Let $G$ be a pro-$p$, consider the cochain complex $C^*_{\rm cts}(G,\ZZ_p)$ of continuous cochains of $G$ with coefficients in the trivial compact $G$ module $\ZZ_p$ as defined by Tate in \cite{Tate1976}. We view this complex as a complex of abstract $\ZZ_p$-modules. Each $C^n_{\rm cts}(G,\ZZ_p)$ is torsion free and thus multiplication by $p$ is a monomorphism of complexes and the cokernel is given by $C^*_{\rm cts}(G,\FF_p)$. We set 
\begin{align*}
    (\BS_r(G),\beta_r)_{r\geq 1}\coloneq(\BS_r(C^*_{\rm cts}(G,\ZZ_p),p),d_r)_{r\geq 1}
\end{align*}
to be the \emph{Bockstein spectral sequence associated to $G$}. We have $\BS_1^i(G)=\rmH^i(G,\FF_p)$.

Note that this spectral sequence is functorial with respect to $G$ and that $\beta^n_1:\rmH^n(G,\FF_p)\to \rmH^{n+1}(G,\FF_p)$ coincides with the Bockstein map defined as the connecting homomorphism in the long exact sequence associated to $0\to \FF_p\to \ZZ/p^2\to \FF_p\to 0$.

In order to compute the maps $\beta_r$ we make use of an explicit presentation of $G$. To that end, we restrict to the case that $G$ is finitely generated, although some of the results will still be valid in the infinitely generated case. 

For a minimal set of generators $x_1,...,x_d$ of $G$, consider the short exact sequence
\begin{equation}
    \label{eq:presentation}
    \begin{tikzcd}
        1\arrow[r]&R\arrow[r]&F\arrow[r,"\pi"]&G\arrow[r]&1
    \end{tikzcd}
\end{equation}
where $F$ is the free pro-$p$ group on $d$-generators and $\pi$ maps these generators to $x_1,...,x_d$ and $R=\ker \pi\subseteq \Phi(F_d)$. 

For a pro-$p$ group $G$ we denote by $P_n(G)$ the $n^{\rm th}$ term of the lower $p$-central series, which is defined by
\begin{align*}
    P_0(G)\coloneq G\qquad \text{and}\qquad P_{n+1}(G)\coloneq P_n(G)^p[P_n(G),G]\quad  \text{for all }n\geq 2.
\end{align*}
In particular $P_2(G)$ coincides with the Frattini subgroup of $G$. Because we will make frequent use of the subgroup $P_3(G)\cap [G,G]=[G,G]^p[[G,G],G]$ we denote it by $P_3^*(G)$. Note that this notation is not standard in the literature.

We start with a simple lemma, which is well known to experts and just a slight variation of \cite{NSW2000}*{Prop. 3.9.13 (i)}:
\begin{lem}
    \label{lem:relation presentation}
    Let $F$ be a free pro-$p$ group generated by $x_1,...,x_d$ and $\rho\in \Phi(F)$, then there exist unique $a_j\in \ZZ_p$ for $j=1,...,d$, $a_{k,l}\in \FF_p$ for $1\leq k<l\leq d$, and $\rho'\in P_3^*(F)$ such that
    \begin{align*}
        \rho=\prod_{j=1}^dx_j^{pa_j}\prod_{1\leq k<l\leq d}[x_l,x_k]^{a_{k,l}}\cdot \rho'
    \end{align*}
\end{lem}
Let $G$ be a finitely presented pro-$p$ group, presented as in \eqref{eq:presentation} with generating set $x_1,...,x_d$. Denote by $\chi_1,...,\chi_d$ the basis of $\rmH^1(G,\FF_p)$ dual to $x_1,...,x_d$. For each $\rho\in R$ there is an associated linear form $\tr_\rho:\rmH^2(G,\FF_p)\to \FF_p$.
\begin{prop}
    \label{prop:HigherBockstein for groups}
    Let $\rho\in R$ and $a_j\in \ZZ_p$, $a_{k,l}$ be as in Lemma~\ref{lem:relation presentation}. If $v_p(a_j)\geq r$ for all $j=1,...,d$ then the map $\tr_{\rho}:\rmH^2(G,\FF_p)\to \FF_p$ induces a linear form $\bar{\tr}_\rho:\BS^2_{r+1}(G)\to \FF_p$.

    If $\alpha\in \BS^1_{r+1}(G)$ is represented by $\sum_{j=1}^d c_j\chi_j\in \rmH^1(G,\FF_p)$, then we have
    \begin{align}
    \label{eq:explicit form of higher Bockstein}
        \bar\tr_{\rho}(\beta_{r+1}^1(\alpha))= -\sum_{i=1}^d c_j \frac{a_j}{p^{r}} 
    \end{align}
\end{prop}
\begin{proof}
    This is done by induction. The case $r=0$ is precisely \cite{NSW2000}*{Prop. 3.9.14}. For the induction step $r-1\to r$ the second statement for $r-1$ yields that $\tr_\rho$ descends to the subquotient $\BS_r^2(G)$ of $\BS_{r-1}^2(G)$, since $\tr_{\rho}$ is zero on the image of $\beta_{r-1}^1$. For the explicit description, one follows the proof of \cite{NSW2000}*{Prop. 3.9.13 (ii)} to construct an explicit cocycle in $\rmH^2_{\rm cts}(G,\ZZ_p)$ representing $\delta(\alpha)$ and arrives at the form of $\beta_r^1(\alpha)$ by the construction of the Bockstein spectral sequence.
\end{proof}
\subsection{The Bockstein spectral sequence for \texorpdfstring{$\ZZ_p$}{Zp}-Lie algebras}
\label{ssec:Bockstein spectral sequence for Lie algebra}
Let $(\Lie,[\bl,\bl])$ be a $\ZZ_p$-Lie algebra, whose underlying $\ZZ_p$-module is free and finitely generated. For a $\Lie$-module $A$ we denote by ${\rm CE}^*(\Lie,A)$ the Chevalley--Eilenberg complex, whose terms are ${\rm CE}^n(\Lie,A)\cong \Hom_{\ZZ_p}(\Lambda^n(\Lie),A)$ and the differentials are given as continuation of 
\begin{align*}
    d^1:{\rm CE}^1(\Lie,A)\to {\rm CE}^2(\Lie,A),\quad f\mapsto (x\wedge y \mapsto xf(y)-yf(x)-f([x,y]))
\end{align*}
For fundamental facts about the Chevalley--Eilenberg complex we refer to \cite{Weibel1994}*{\S 7.7}. The most important fact is that it computes Lie algebra cohomology, i.e., $\rmH^i({\rm CE}^*(\Lie,A))\cong \rmH^i(\Lie,A)$. 

Consider the following short exact sequence of complexes:
\begin{equation*}
    \begin{tikzcd}
        0\arrow[r]&{\rm CE}^*(\Lie,\ZZ_p)\arrow[r,"p"]&{\rm CE}^*(\Lie,\ZZ_p)\arrow[r]&{\rm CE}^*(\Lie,\FF_p)\arrow[r]& 0
    \end{tikzcd}
\end{equation*}
The \emph{Bockstein spectral sequence associated to $\Lie$} is defined as
\begin{align*}
    (\BS_r(\Lie),\mathfrak{b}_r)_{r\geq 1}\coloneq(\BS_r({\rm CE}^*(\Lie,\ZZ_p),p),d_r)_{r\geq 1}.
\end{align*}
Similarly, as before, we have $\BS_1^i(\Lie)\cong \rmH^i(\Lie,\FF_p)$ and the spectral sequence is functorial in $\Lie$. 

From now on we assume $\Lie$ to be \emph{powerful}, i.e., $[\Lie,\Lie]\subseteq p\Lie$ (cf., e.g., \cite{DDMS1999}*{\S 9.4}). In this situation $\rmH^i(\Lie,\FF_p)\cong \Lambda^i((\Lie\otimes \FF_p)^\vee)$, where $\bl^\vee$ denotes the $\FF_p$-linear dual. 

By the torsion-freeness of $\Lie$ as $\ZZ_p$-module we can construct another Lie bracket $\dbb{\bl,\bl}$ on $\Lie$ by the property that for any $x,y\in \Lie$ we have $p\dbb{x,y}=[x,y]$. We denote $\Lie$ equipped with this new bracket by $\fLie$. Using this new Lie bracket we can give an explicit description of the connecting homomorphism $\rmH^1(\Lie,\FF_p)=\Hom_{\ZZ_p}(\Lie,\FF_p)\to \rmH^2(\Lie,\ZZ_p)$ by 
\begin{align*}
    f\mapsto [(x\wedge y \mapsto -\widehat{f}(\dbb{x,y}))]
\end{align*}
where $\widehat{f}:\Lie\to \ZZ_p$ is any lift of $f$. Notice that $\widehat{f}$ is not a homomorphism, unless $f=0$. It follows by a standard argument that the cohomology class in $\rmH^2(\Lie,\ZZ_p)$ is independent of the choice of lift. 

Using this description we derive a statement similar to Proposition~\ref{prop:HigherBockstein for groups} for powerful $\ZZ_p$-Lie algebras. For our later applications it is of advantage to consider the duals of $\mathfrak{b}^1_r$, rather than $\mathfrak{b}^1_r$ itself.
\begin{prop}
\label{prop:higherBockstein Lie algebra}
    Let $\lambda\in \Lambda^2(\Lie\otimes \FF_p)=\BS_1^2(\Lie)^\vee$, then $\lambda$ defines an element in $\BS_r^2(\Lie)$ if and only if there exists a lift $\widehat{\lambda}\in \Lambda^2(\Lie)$ such that $\dbb{\widehat{\lambda}}\equiv 0\pmod {p^{r-1}}$. In this case 
    \begin{align*}
        (\mathfrak{b}_r^1)^\vee([\lambda])=-\big[{\dbb{\widehat{\lambda}}}/{p^{r-1}}\big]\in \BS^1_1(\Lie)^\vee
    \end{align*}
\end{prop}
\begin{proof}
    The proof works exactly in the same way as the one of Proposition~\ref{prop:HigherBockstein for groups}.
\end{proof}
The next corollary is an immediate consequence of Proposition~\ref{prop:higherBockstein Lie algebra}
\begin{cor}
\label{cor:Second page Bockstein Lie algebra}
    We have $\BS_2^1(\Lie)^\vee\cong (\fLie_p)^{\rm ab}$ and  $\mathfrak{b}^1_s=0$ for all $s<r$ if and only if $\dbb{\fLie,\fLie}\subseteq p^{r-1}\fLie$.
\end{cor}
\subsection{Uniform pro-\texorpdfstring{$p$}{p} groups}
\label{ssec:Bockstein Spectral Sequences are the same for uniform}
The theory of $p$-adic analytic groups, Lie groups over the field $\QQ_p$, has its origins in the seminal paper of M.~Lazard~\cite{Lazard1965}. A.~Lubotzky and A.~Mann simplified and reinterpreted in \cites{LubotzkyMann1987I,LubotzkyMann1987II} Lazard's ideas and introduced (uniformly) powerful pro-$p$ groups yielding a more intrinsic group-theoretic framework for studying $p$-adic analytic groups. An excellent exposition of this topic is given by J.~D.~Dixon, M.~P.~F.~du Sautoy, A.~Mann and D.~Segal in  \cite{DDMS1999}.

A \emph{uniformly powerful} or \emph{uniform pro-$p$ group} is a finitely generated pro-$p$ group $U$, for which $[U,U]\subseteq U^p$ (and $[U,U]\subseteq U^4$ if $p=2$), which is torsion free (cf. \cite{DDMS1999}*{Def. 4.1 \& Thm. 4.5}).

\begin{fct}
    Every powerful group (i.e. a pro-$p$ group $G$ with $[G,G]\subseteq G^p$) has an open uniform subgroup and every compact $p$-adic analytic group contains an open uniform subgroup.
\end{fct}
One of the principal constructions for uniform pro-$p$ groups $U$ is the associated powerful $\ZZ_p$-Lie algebra $\Lie(U)$ (see \cite{DDMS1999}*{\S 4.5}). This construction is inspired by work of Lazard about $p$-saturated groups. There are mutually inverse maps $\exp:\Lie(U)\to U$ and $\log:U\to \Lie(U)$ such that 
\begin{align*}
    \exp(x)\exp(y)=\exp(\Phi(x,y))
\end{align*}
where $\Phi$ is the Baker--Campbell--Hausdorff series. 

The following theorem shows the strength of the theory of uniform pro-$p$ groups:
\begin{thm}
    There is an equivalence between the category of uniform pro-$p$ groups and the category of powerful Lie algebras over $\ZZ_p$, i.e., Lie algebras $\Lie$ over $\ZZ_p$ that are free as $\ZZ_p$-modules and satisfy $[\Lie,\Lie]\subseteq p\Lie$.
\end{thm}
In \cite{HuberKingsNaumann2011} A.~Huber, G.~Kings, and N.~Naumann proved an integral version of the Lazard isomorphism theorem, which originates in Lazard's seminal work \cite{Lazard1965} on $p$-adic analytic groups. {\color{black} The following lemma is not explicitly stated in their work, but not difficult to deduce from their results.
\begin{lem}
    Let $U$ be a uniform group and $\Lie=\Lie(U)$ the Lie algebra associated to $U$ and $M$ a finitely generated $\ZZ_p$-module, considered as a trivial $U$ and $\Lie$-module. Then for every $n$ there is an isomorphism $\phi^n_G(M):\rmH^n_{\rm cts}(G,M)\to \rmH^n(\Lie,M)$ natural in $M$. Such that for a short exact sequence $0\to M'\to M\to M''\to 0$ of finitely generated trivial $G$ modules the following square commutes:
    \begin{equation*}
        \begin{tikzcd}
            \rmH^n_{\rm cts}(G,M'')\arrow[r,"\delta"]\arrow[d,"\phi^{n}_G(M')"]&\rmH_{\rm cts}^{n+1}(G,M'')\arrow[d,"\phi^{n+1}_G(M')"]\\
            \rmH^n(\Lie,M'')\arrow[r,"\delta"]&\rmH^{n+1}(\Lie,M')
        \end{tikzcd}
    \end{equation*}
\end{lem}
\begin{proof}
    Uniform groups are equi-$p$-valued in the sense of \cite{Lazard1965} with respect to the valuation associated to the lower $p$-central series (see beginning of the proof of \cite{HuberKingsNaumann2011}*{Thm. 3.3.3}) taking values in $\ZZ$. As the action of $G$ on $M$ is trivial, the homomorphism $\ZZ_p[\![G]\!]\to \operatorname{End}_{\ZZ_p}(M)$ factors through the augmentation map $\ZZ_p[\![G]\!]\to \ZZ_p$ and hence extends to $\operatorname{Sat}(\ZZ_p[\![G]\!])$ (see \cite{Lazard1965}*{Chap. I,\S 2.2.11}). Thus \cite{HuberKingsNaumann2011}*{Thm. 3.1.1} is applicable, yielding the required homomorphisms and their naturality in $M$. Let $0\to M'\to M\to M''\to 0$ be an exact sequence of finitely generated trivial $G$-modules. Then by \cite{Tate1976} there is a long exact sequence in continuous cochain cohomology. The construction of $\phi^n_G(\bl )$ on the level of cochain complexes implies that the square commutes, as claimed. By \cite{HuberKingsNaumann2011}*{Prop. 3.4.8} these complexes are quasi isomorphic. Hence the short exact sequence of modules yields a short exact sequence of the complexes, showing that it commutes with the connecting homomorphisms.
\end{proof}
} An immediate consequence of this generalization is the following proposition, which will help us to reduce a problem in the cohomology of uniform groups to a problem in linear algebra.
\begin{thm}
\label{thm:uniform Bocksteins agree}
    Let $U$ be a uniform pro-$p$ group and $\Lie(U)$ be its associated $\ZZ_p$-Lie algebra. Then the Bockstein spectral sequences $(\BS^*_r(U),\beta_r^*)$ and $(\BS^*_r(\Lie(U)),\mathfrak{b}_r^*)$ are naturally isomorphic.
\end{thm}

\section{Uniform quotients and groups with the uFM property}
\label{sec:Uniform quotients}
Having established the comparison between Bockstein spectral sequences of uniform groups and Lie algebras, we now use these tools to study uniform quotients. Recall from the introduction, that a pro-$p$ group $G$ is said to have the \emph{uniform Fontaine--Mazur property} (or short \emph{uFM property}) if $G$ does not admit a nontrivial uniform pro-$p$ quotient. 
\subsection{uFM groups and Bockstein spectral sequences}
Let $G$ be a pro-$p$ and $\pi:G\twoheadrightarrow U$ be the projection on a uniform pro-$p$ group and $\Lie$ be the Lie algebra associated to $U$. Then by Theorem~\ref{thm:uniform Bocksteins agree} and the functoriality of $\BS_r(\bl)$ we have the following commutative diagram for each $r\in \NN$:
\begin{equation}
\label{diag:Bockstein square}
    \begin{tikzcd}
        \BS_r^1(G)\arrow[d,"\beta_r^1"] & \BS_r^1(U)\arrow[l,swap,"\pi^*"]\arrow[d,"\beta_r^1"]&\BS_r^1(\Lie)\arrow[l,swap,"\sim"]\arrow[d,swap,"\mathfrak{b}_r^1"]\\
        \BS_r^2(G) & \BS_r^2(U)\arrow[l,swap,"\pi^*"]&\BS_r^2(\Lie)\arrow[l,swap,"\sim"]
    \end{tikzcd}  
\end{equation}
We consider the outer square and take linear duals. For $r=1$ this yields the following commutative square:
\begin{equation*}
    \begin{tikzcd}
            \rmH^1(G,\FF_p)^\vee\arrow[r,equal]&G/\Phi(G)\arrow[r,two heads]&\fLie_p\\
             \arrow[r,"\inf^\vee"]\rmH^2(G,\FF_p)^\vee \arrow[u,swap,"(\beta^1_1)^\vee"]&\rmH^2(U,\FF_p)^\vee\arrow[r,"(\bl\smallsmile \bl)^\vee"]&\Lambda^2(\fLie_p)\arrow[u,"-\dbb{\bl}_p"]
        \end{tikzcd}
\end{equation*}
Let $x_1,...,x_d$ be a minimal generating set of $G$ and set $e_i\coloneq \log(\pi(x_i))\in \Lie$. If $\rho$ is a relation with $\rho\equiv \prod_{j=1}^dx_j^{pa_j}\prod_{1\leq k<l\leq d}[x_l,x_k]^{a_{k,l}}\pmod{P_3^*(F)}$, as in Lemma~\ref{lem:relation presentation}. Then Propositions~\ref{prop:HigherBockstein for groups} and \ref{prop:higherBockstein Lie algebra} applied to the element $\tr_\rho\in \rmH^2(G,\FF_p)^\vee$ yield together with the commutativity of the diagram the following congruence in $\fLie$:
\begin{align}
\label{eq:Bockstein 1 relation}
    \sum_{j=1}^d a_j e_j\equiv \sum_{k<l}a_{k,l}\dbb{e_k,e_l}\pmod p
\end{align}
This relation is well known, when working with the graded Lie algebra associated to $p$-lower central series instead of with the $\ZZ_p$-Lie algebra associated to a uniform group. The translation is immediate, as this graded Lie algebra can be obtained from $\Lie$ using the filtration $(p^n\Lie)_n$. This provides a link between ours and J.~Labute's approach in~\cite{Labute2014}.

When the valuation of the $a_i$ is at least $1$, the higher Bockstein morphisms give a way to lift the relations modulo higher powers of $p$. This is made precise in the subsequent two propositions.
\begin{prop}
\label{prop:Bockstein 2 relation}
    If for all $j=1,..,d$ we have $a_j\in p\ZZ_p$, then the following congruence holds for any lift $\widehat{a}_{k,l}\in \ZZ_p$ of $a_{k,l}$:
    \begin{align*}
        \sum_{j=1}^d a_j e_j\equiv \sum_{k<l}\widehat{a}_{k,l}\dbb{e_k,e_l}\mod{(p^2,p\dbb{\fLie,\fLie})}
    \end{align*}
\end{prop}
\begin{proof}
    This follows in the same manner as before from the commutative diagram (\ref{diag:Bockstein square}) together with Propositions~\ref{prop:HigherBockstein for groups} and \ref{prop:higherBockstein Lie algebra} and Corollary~\ref{cor:Second page Bockstein Lie algebra}.
\end{proof}
\begin{prop}
\label{prop:First r Bocksteins vanish}
    Assume that $\dbb{\fLie,\fLie}\subseteq p^{r-1}\fLie$ for some $r$. If $a_j\in p^{r-1}\ZZ_p$ for all $j=1,...,d$, then we have for any lift $\widehat{a}_{k,l}$ of $a_{k,l}$ to $\ZZ_p$
    \begin{align*}
        \sum_{j=1}^d a_j e_j\equiv \sum_{k<l}\widehat{a}_{k,l}\dbb{e_k,e_l}\pmod{p^r}.
    \end{align*}
    If $s\coloneq\min v_p(a_j)<r-1$, then the element in $G^{\rm ab}$ given by $\prod_{j=1}^d \overline{x}_j^{a_j}$ defines an element in the kernel of the induced map
    \begin{align*}
        G^{\rm ab}\otimes \ZZ/p^s\to U^{\rm ab}\otimes \ZZ/p^s\cong (\ZZ/p^s)^{d(U)}.
    \end{align*}
\end{prop}
\begin{proof}
    We have that $\BS_r^1(\Lie)\cong \fLie_p$ by applying Proposition~\ref{prop:higherBockstein Lie algebra}. The same arguments as before yield the first statement. The second one follows since $(\beta_s^1)^\vee(\tr_\rho)$ is contained in the kernel of $\BS_s^{1}(G)^\vee \to \BS_s^1(\Lie)^\vee\cong \fLie_p$, as otherwise $\BS_{s+1}^1(\Lie)\not\cong \fLie_p$. Translating this fact back to the group-theoretic context gives the second statement.
\end{proof}

\subsection{An example with four generators}
{\color{black} Without posing further assumptions on $\Lie^{\flat}$ it is hard to satisfy the conditions of Proposition~\ref{prop:First r Bocksteins vanish}. The following example shows that for suitable presentations, this can be achieved without any additional assumptions on a uniform quotient $U$. }
\begin{exmp}
\label{exmp:four generators five relations}
    Let $\langle x_1,...,x_4\mid \rho_1,...,\rho_5\rangle $ be the presentation of a pro-$p$ group $G$ with 
    \begin{align*}
        \rho_1\equiv x_1^{p^n}[x_1,x_2],\quad \rho_2\equiv x_2^{p^n}[x_2,x_3],\quad \rho_3\equiv x_3^{p^n}[x_3,x_4]\\
        \quad \rho_4\equiv x_4^{p^n}[x_4,x_1],\quad \text{and}\quad  \rho_5\equiv [x_1,x_3][x_2,x_4]
    \end{align*}
    where all the congruences are taken modulo $P_3^*(F)$. Then each uniform quotient of $G$ is trivial.
    
    Let $\pi:G\twoheadrightarrow U$ be the projection onto a uniform quotient, $\Lie=\Lie(U)$, and $e_i=\log(\pi(x_i))$. For $n=0$ the claim is clear. {\color{black} For $n=1$ we set $a=\dbb{e_1,e_3}\equiv -\dbb{e_2,e_4}\pmod p$ and find
    \begin{align*}
        e_1\equiv \dbb{e_1,e_2}\equiv \dbb{e_1,\dbb{e_2,e_3}}\equiv \dbb{e_2,a}+\dbb{e_1,e_3}\equiv \dbb{e_2,a}+a\pmod p
    \end{align*}
    and similarly $e_2\equiv -\dbb{e_3,a}-a,\, e_3\equiv \dbb{e_4,a}+a,\, e_4\equiv \dbb{e_1,a}+a$.
    The Jacobi identity for $a,e_1$ and $e_3$ yields
    \begin{align*}
        0&\equiv \dbb{a,\dbb{e_1,e_3}}\equiv \dbb{\dbb{a,e_1},e_3}+\dbb{e_1,\dbb{a,e_3}}\equiv \dbb{a-e_4,e_3}+\dbb{e_1,e_2+a}\\
        &\equiv e_3+e_2+a+e_1+e_4-a\equiv e_1+e_2+e_3+e_4\pmod p
    \end{align*}
    Substituting this into the original congruences we find
    \begin{align*}
        e_1&\equiv \dbb{e_1,e_2}\equiv \dbb{-e_3,e_2}+\dbb{-e_4,e_2}\equiv e_2-a\quad\text{and}\\
        e_3&\equiv \dbb{e_3,e_4}\equiv \dbb{-e_1,e_4}+\dbb{-e_2,e_4}\equiv e_4+a.
    \end{align*}
    yielding the additional relation $e_1-e_2+e_3-e_4\equiv 0\pmod p$. This shows that $\dim \fLie_p\leq 2$. But $\fLie_p$ is perfect. This is a contradiction. 
    
    For $n\geq 2$ we have $0\equiv \dbb{e_i,e_{i+1}}\pmod p$ for $i=1,..,4$ and $\dbb{e_1,e_3}\equiv -\dbb{e_2,e_4}\pmod p$. We write $\dbb{e_1,e_3}=\lambda_1e_1+...+\lambda_4e_4$ for suitable $\lambda_i\in \ZZ_p$. Then the Jacobi identity implies
    \begin{align*}
        0\equiv \dbb{e_1,\dbb{e_2,e_3}}\equiv \dbb{\dbb{e_1,e_2},e_3}+\dbb{e_2,\dbb{e_1,e_3}}\equiv \lambda_4\dbb{e_2,e_4}\equiv -\lambda_4\dbb{e_1,e_3}.
    \end{align*}
    Thus if $\dbb{e_1,e_3}\not\equiv 0\pmod p$, then $\lambda_4=0$. The same argument applies for the other coefficients and sees that $\dbb{e_1,e_3}\equiv \dbb{e_2,e_4}\equiv 0\pmod p$ and thus $\fLie_p$ is abelian and we conclude $\dbb{\fLie,\fLie}\subseteq p^1\fLie$. One proceeds by induction using Proposition~\ref{prop:First r Bocksteins vanish} to see that $\dbb{\fLie,\fLie}\subseteq p^{n-1}\fLie$. The claim then follows from the case $n=1$.}
\end{exmp}

\section{Koch-type groups with three generators}
\label{sec:Koch type groups with three generators}
We now specialize the general method developed in Section \ref{sec:Uniform quotients} to Koch-type groups with three generators corresponding to the groups $G_{K,S}$ with $|S|=3$. {\color{black} We start by recalling some well-known results on $\ZZ_p$-Lie algebras with finite abelianization and prove Proposition~\ref{prop:FAb Lie algebra dim3}, which classifies those of dimension $3$. Using this result, we present a method to study pro-$p$ groups, which are weakly presented by a linking diagram with three vertices. Finally we prove the infinitude of some $G_{K,S}$ with $|S|=3$, that also have the uFM property.}
\subsection{three-dimensional \texorpdfstring{$\ZZ_p$}{Zp}-Lie algebras with finite abelianization}
A reference for the first statement can be found in \cite{HajirMaire2022}*{Prop. 3.18 and Cor. 3.19} and for the second in \cite{Maire2018}*{Cor. 2.12}.
\begin{prop}
\label{prop:FAB uniform groups}
    Let $U$ be a uniform pro-$p$ group and $\Lie$ its associated $\ZZ_p$-Lie algebra. Then the following are equivalent:
    \begin{center}
        $U^{\rm ab}$ is finite
        $\Leftrightarrow$ $U$ is FAb
        $\Leftrightarrow$ $\Lie$ is FAb
        $\Leftrightarrow$ $\Lie^{\rm ab}$ is finite
        $\Leftrightarrow$ $\Lie\otimes_{\ZZ_p} \QQ_p$ is perfect
    \end{center}
    Furthermore, if $U$ is nontrivial and FAb, then $d(U)\geq 3$.
\end{prop}
We shall restrict ourselves to $\ZZ_p$-Lie algebras of dimension $3$, i.e., $\Lie\cong \ZZ_p^3$ as $\ZZ_p$ modules, with finite abelianization. We believe that the following proposition is well known, but we have not been able to find any account of it in the literature.
\begin{prop}
\label{prop:FAb Lie algebra dim3}
    Let $\Lie$ be a Lie algebra of dimension $3$ for $p$ odd. Fix a representative $\beta$ of the nontrivial element of $\ZZ_p^\times /\ZZ_p^{\times 2}\cong \ZZ/2$. If $\Lie^{\mathrm{ab}}$ is finite, then there exists a basis $e_1,e_2,e_3$ of $\Lie$, non-negative integers $v_1,v_2,v_3\in \mathbb{N}_0$ and $s\in \{0,1\}$ such that
    \begin{align*}
        [e_1,e_2]=\beta^{s}p^{v_3}e_3,\quad [e_3,e_1]=p^{v_2}e_2,\quad [e_2,e_3]=p^{v_1}e_1
    \end{align*}
\end{prop}
\begin{proof}
\color{black}
    The proof essentially follows \cite{Jacobson1962}*{I.4 (e)}. Pick a basis $\{e_1,e_2,e_3\}$ of $\Lie$ as free $\ZZ_p$-module. We set $f_1\coloneq [e_2,e_3]$, $f_2\coloneq [e_3,e_1]$ and $f_3\coloneq [e_1,e_2]$ and write 
    \begin{align*}
        f_i={\sum}_{j=1}^3\alpha_{ij}e_j
    \end{align*}
    for suitable $\alpha_{ij}\in \ZZ_p$. This yields a matrix $A=(\alpha_{ij})\in \ZZ_p^{3\times 3}$, which is symmetric by the Jacobi identity. Since $\Lie\otimes_{\ZZ_p}\QQ_p$ is perfect, it follows that $\det A\neq 0$. Similar to the proof of \cite{Jacobson1962}*{I.4 (e)} there is a basis $\{e_1',e_2',e_3'\}$ of $\Lie$ with associated matrix $A'$ if and only if $A'=\lambda N^TAN$ for suitable $\lambda\in \ZZ_p^\times$ and $N\in \GL_3(\ZZ_p)$.

    This type of equivalence relation is rather well studied in the theory of quadratic forms. From \cite{OMeara1973}*{\S 92.} it follows that there is a matrix $N\in \GL_3(\ZZ_p)$ such that $N^TAN$ is diagonal. We write the diagonal entries as $\beta^{s_i}u_i^2p^{v_i}$ for $i=1,2,3$, $u_i\in \ZZ_p^\times$ and $s_i\in\{0,1\}$. By multiplying with $\lambda\in \{1,\beta\}$ we can assume that $s_i=0$ for all $i$ or $s_i=1$ for exactly one $i$, which by permuting of the entries can be put in the third position. We have
    \begin{align*}
    A\sim 
        \medmatrix{
            u_1^{-1}&0&0\\
            0&u_2^{-1}&0\\
            0&0&u_3^{-1}
        }^T
        (\lambda N^TAN)
        \medmatrix{
            u_1^{-1}&0&0\\
            0&u_2^{-1}&0\\
            0&0&u_3^{-1}
        }=\medmatrix{
            p^{v_1}&0&0\\
            0&p^{v_2}&0\\
            0&0&\beta^{s}p^{v_3}
        }
    \end{align*}
    By picking the associated $\ZZ_p$-basis of $\Lie$, the desired statement follows.
\end{proof}
A consequence of Proposition~\ref{prop:FAb Lie algebra dim3} is that, in special cases, it is possible to deduce from the isomorphism class of a three-dimensional $\FF_p$-Lie algebra, whether it is the reduction of a FAb $\ZZ_p$-Lie algebra. The following corollary makes this precise.
\begin{cor}
\label{cor:Mod p reduction of FAb LieAlgebras dim3}
    Let $\Lie$ be a three-dimensional $\ZZ_p$-Lie algebra with finite abelianization. We set $\Lie_p\coloneq \Lie\otimes \FF_p$. Then the following hold:
    \begin{enumerate}
        \item \label{it:one dimensional derived implies central} If $\dim_{\FF_p}\mathfrak{g}_p'=1$, then $\Lie_p'\subseteq \mathfrak{z}(\Lie_p)$. In particular, all such Lie algebras are isomorphic.
        \item \label{it:two dimensional derived implies iso class} If $\dim_{\FF_p}\mathfrak{g}_p'=2$, then the conjugacy class in $\mathrm{PGL}_2(\FF_p)$ associated to $\mathfrak{g}_p$ is generated by 
        \begin{align*}
            \medmatrix{
                0&-1\\
                1&0
            }
            \qquad \text{or}\qquad 
            \medmatrix{
                0&-\beta\\
                1&0
            }
        \end{align*}
        where $\beta$ is a non-square in $\FF_p$ and there are only two possible isomorphism classes.
    \end{enumerate}
\end{cor}
\begin{proof}
    Using Proposition~\ref{prop:FAb Lie algebra dim3} one chooses a suitable basis and notices that $\Lie_p'$ is spanned by the reductions of the $p^{v_i}e_i$. The number of $v_i$ which are $0$ is the dimension of $\Lie_p'$, from which the claim follows by studying the relations in $\Lie$.
\end{proof}
{\begin{exmp}
    Let $\Lie=\FF_pe\oplus \FF_pf\oplus \FF_pg$, then the definitions
    \begin{align*}
        [e,f]=e&,\quad [e,g]=0,\quad [f,g]=0,\qquad \text{and respectively}\\
        [e,f]=0&,\quad [e,g]=e+f,\quad [f,g]=f
    \end{align*}
    define Lie algebra structures on $\Lie$ (see \cite{Jacobson1962}*{I \S4 (c) \& (d)}). By Corollary~\ref{cor:Mod p reduction of FAb LieAlgebras dim3} neither of these Lie algebras is the mod-$p$ reduction of a three-dimensional $\ZZ_p$-Lie algebra with finite abelianization.
\end{exmp}
\begin{quest}
    What are necessary conditions on an $\FF_p$-Lie algebra (in dimension $>3$) to be the mod-$p$ reduction of a FAb $\ZZ_p$-Lie algebra of the same dimension (in particular free as a $\ZZ_p$-module)?
\end{quest}}
\subsection{Pro-\texorpdfstring{$p$}{p} groups and linking diagrams with three vertices}
\label{ssec:FAb groups with three generators}
Let $G$ be a pro-$p$ group with $d(G)=3$ weakly presented by a linking diagram $\Lambda=(\Gamma,a,\ell)$. We write $V(\Gamma)=\{1,2,3\}$ and $\ell_{i,j}=\ell(i,j)$ to simplify the notation. Furthermore, we assume that $v_p(a_1)\leq v_p(a_2)\leq v_p(a_3)$. Let us collect all the relevant data in the following way:
\begin{align}
\label{eq:combinatorial data}
    A\coloneq\medmatrix{
        a_1&0&0\\ 0&a_2&0\\ 0&0&a_3 
    }\in \ZZ_p^{3\times 3}\qquad 
    L\coloneq\medmatrix{
        0&-\ell_{13}&\ell_{12}\\
        \ell_{23}&0&-\ell_{21}\\
        -\ell_{32}&\ell_{31}&0
    }\in \FF_p^{3\times 3}
\end{align}
We write $A_p$ for the mod-$p$ reduction of $A$.
\begin{rem}
    If $\dim \rmH^2(G,\FF_p)=3$, i.e. $G$ is presented by $\Lambda$ in the strong sense, then the matrix $L$ describes the cup product $\Lambda^2(\rmH^1(G,\FF_p))\to \rmH^2(G,\FF_p)$ with respect to the appropriate bases of $\Lambda^2(\rmH^1(G,\FF_p))$ and $\rmH^2(G,\FF_p)$. Thus one has $\rk L=\dim_{\FF_p}(\rmH^1(G,\FF_p)\smallsmile \rmH^1(G,\FF_p))$ and $\rk L$ is an invariant of the group and independent of the specific chosen presentation.

{\color{black}
    Furthermore, the valuations of the $a_i$ are also invariants of the group, as they can be read of the elementary divisors of the abelianization.} 
\end{rem}
Assume that $G$ admits a nontrivial uniform quotient $\pi:G\twoheadrightarrow U$, let $\Lie=\Lie(U)$ be its Lie algebra and $e_i=\log(\pi(x_i))$. Then we set 
\begin{align*}
    f_1\coloneq \dbb{e_2,e_3},\quad f_2\coloneq \dbb{e_3,e_1},\quad \text{and}\quad f_3\coloneq \dbb{e_1,e_2}.
\end{align*}
With notations fixed this way, (\ref{eq:Bockstein 1 relation}) reads as 
\begin{align}
\label{eq:relations in three dim Lie algebra}
    A\begin{pmatrix}
        e_1,e_2,e_3
    \end{pmatrix}^T\equiv
    L\begin{pmatrix}
        f_1,f_2,f_3
    \end{pmatrix}^T
    \pmod p.
\end{align}
To test whether $G$ satisfies all the necessary conditions to admit a nontrivial uniform quotient we {\color{black} use the results from Lemma~\ref{lem:rank and valuation relation} and Propositions~\ref{prop:rkL = 3} to \ref{prop:rkL = 1}. The method to check their conditions is summarized in Table~\ref{tab:rk3 Algorithm}}. It distinguishes $16$ cases by the rank of $L$ and the number of zero entries in the triple $(v_1,v_2,v_3)=(v_p(a_1),v_p(a_2),v_p(a_3))\in \NN_0^3$. By $*$ we indicate any positive integer. The lightning symbol ($\lightning$) indicates that a $G$ with the given combinatorial data cannot admit a nontrivial uniform quotient. If the method is not sufficient to exclude the existence of a nontrivial uniform quotient, we write \emph{inconclusive}.
\begin{center}
\begin{table}[h!]
    \begin{tabular}{c||K{2cm}|K{2cm}|K{2cm}|K{2cm}}
         $(v_1,v_2,v_3)$&  $(0,0,0)$ & $(0,0,*)$ & $(0,*,*)$ & $(*,*,*)$\\
         \hline \hline 
        \makecell{$\rk(L)=3$\\ see Prop.~\ref{prop:rkL = 3} }&\makecell{ check\\ Prop.~\ref{prop:rkL = 3} (\ref{it:rkL = 3 a})} & $\lightning$ & $\lightning$ & divide each $a_i$ by $p^{v_1}$ and start over\\
        \hline 
        \makecell{$\rk(L)=2$\\ see Prop.~\ref{prop:rkL = 2} }& $\lightning$ & \makecell{ check\\ Prop.~\ref{prop:rkL = 2} (\ref{it:rkL = 2 a})} &  \makecell{ check\\ Prop.~\ref{prop:rkL = 2} (\ref{it:rkL = 2 b})} & \makecell{ check\\ Prop.~\ref{prop:rkL = 2} (\ref{it:rkL = 2 c})}\\
        \hline 
        \makecell{$\rk(L)=1$\\ see Prop.~\ref{prop:rkL = 1} }& $\lightning$ & $\lightning$ & \makecell{check\\ Prop.~\ref{prop:rkL = 1} (\ref{it:rkL = 1 a})} & {inconclusive}\\
        \hline 
        $\rk(L)=0$ & $\lightning$ & $\lightning$ & $\lightning$ & inconclusive\\
    \end{tabular}
    \caption{Summary of the method to determine whether $G$ represented by $\Lambda$ with $A$ and $L$ given by \eqref{eq:combinatorial data} has the uFM property}
    \label{tab:rk3 Algorithm}
    \end{table}
\end{center}
Before stating Propositions~\ref{prop:rkL = 3}, \ref{prop:rkL = 2} and \ref{prop:rkL = 1}, we want to highlight a consequence of the existence of this method. Since it relies only on the combinatorial data, it is possible to artificially construct suitable linking diagrams with three vertices satisfying the conditions in each case. By an application of Theorem~\ref{thm:Linking diagram approximation by arithmetic} we can realize this linking diagram by infinitely many triples of primes and conclude the following theorem.
\begin{thm}
\label{thm:realization of uFM groups}
    Fix a number field $K$ and a prime $p$ unramified in $K/\QQ$ with $h_p(K)=1$. Let $a_1,a_2,a_3\in \ZZ\setminus \{0\}$, $k\in \NN$, and $0\leq r\leq 3$ with $r\geq 2$ if each $a_i$ is divisible by $p^2$. Then there exist infinitely many triples of tame places $S=\{v_1,v_2,v_3\}$ such that $G_{K,S}$ has the uFM property, $N(\mathfrak{p}_{v_i})-1\equiv pa_i\pmod {p^k}$ and $\rk L=r$.
\end{thm}
\begin{proof}
    The only case where the above method (more specifically Proposition~\ref{prop:rkL = 2}) does not give the answer is when $\rk(L)=2$ and each $a_i$ is divisible by $p^2$. In this case we refer to Example~\ref{exmp:rk2 and Lie algebra not abelian}, which gives a family having the uFM property by choosing $x_3$ in the example to correspond to the $a_i$ with minimal valuation.
\end{proof}
The next lemma, we justify, why Table~\ref{tab:rk3 Algorithm} has a triangular form. 
\begin{lem}
    \label{lem:rank and valuation relation}
    If $\sum_{j=1}^3c_jL_{j-}=0$ for some $c_j\in \FF_p$, then $c_j=0$ for all $j$ with $v_p(a_j)=0$. In particular, the rows $L_{j-}$ for all $j$ with $v_p(a_j)=0$ are linearly independent and the rank of $L$ is at least the number of $j$s such that $v_p(a_j)=0$.
\end{lem}
\begin{proof}
    Let $\sum_{j=1}^3c_jL_{j-}=0$. Then it follows from \eqref{eq:relations in three dim Lie algebra} that
    \begin{align*}
        \sum_{j=1}^3(c_ja_j)e_i\equiv  \bigg(\sum_{j=1}^3c_jL_{j-}\bigg)(f_1,f_2,f_3)^T\equiv 0\pmod p.
    \end{align*}
    From the linear independence of the $e_i$ it follows that $c_ja_j\equiv 0\pmod p$ for all $j$. Hence there can be no $j$ with $v_p(a_j)=0$ and $0\neq c_j\in \FF_p$.

    The last two statements follow from the first one.
\end{proof}
The following propositions yield the necessary conditions indicated in Table~\ref{tab:rk3 Algorithm}. Since the proofs are very technical and not very enlightening we decided to move them to Appendix~\ref{sec:Appendix with proofs}.
\begin{prop}
    \label{prop:rkL = 3}
    If $\rk L=3$ and there exists a nontrivial uniform quotient of $G$, then either
    \begin{enumerate}[(a)]
        \item \label{it:rkL = 3 a} each $v_i$ is $0$ and $A_p^{-1}L$ is symmetric or
        \item \label{it:rkL = 3 b} each $v_i> 0$ and there is a three-dimensional FAb $\ZZ_p$-Lie algebra $\widetilde{\Lie}$ such that $\widetilde{\Lie}_p$ satisfies the relations (\ref{eq:relations in three dim Lie algebra}) with $a_i$ replaced by $a_i/p^{v_1}$.
    \end{enumerate}  
\end{prop}
\begin{prop}
    \label{prop:rkL = 2}
    If $\rk L=2$ and there exists a nontrivial uniform quotient of $G$, then either
    \begin{enumerate}[(a)]
        \item \label{it:rkL = 2 a} $v_1=v_2=0$, $v_3> 0$ and the rows $L_{1-}$, $L_{2-}$ are linearly independent, $L_{3-}=0$, and both $\ell_{13}$ and $\ell_{23}$ non-zero.
        \item \label{it:rkL = 2 b} $v_1=0$, $v_2,v_3> 0$ and $L_{1-}$ is non-zero, $L_{2-}$ and $L_{3-}$ are linearly dependent, and $L_{21}=L_{31}=0$.
        \item \label{it:rkL = 2 c} $v_1,v_2,v_3> 0$. In this case let $0\neq \mu=(\mu_1,\mu_2,\mu_3)\in \FF_p^3$ such that $\mu L=0$ and set
        \begin{align*}
            M\coloneq\medmatrix{
                0&(a_3/p) \mu_3&-(a_2/p) \mu_2\\
                -(a_3/p) \mu_3&0&(a_1/p) \mu_1\\
                (a_2/p) \mu_2&-(a_1/p) \mu_1&0
            }
            \in \FF_p^{3\times 3}\quad \text{and}\quad \widehat{L}\coloneq \medmatrix{
                L\\
                M
            }
        \end{align*}
        If the matrix $\widehat{L}$ has rank $3$, then there exists a three-dimensional FAb $\ZZ_p$-Lie algebra $\widetilde{\Lie}$ such that $\widetilde{\Lie}_p$ satisfies the relations (\ref{eq:relations in three dim Lie algebra}) with $a_i$ replaced by $a_i/p$.
    \end{enumerate}  
\end{prop}
\begin{prop}
\label{prop:rkL = 1}
    If $\rk L=1$ and there exists a uniform quotient of $G$, then either
    \begin{enumerate}[(a)]
        \item \label{it:rkL = 1 a} $v_1=0$, $v_2,v_3>0$ and the rows $L_{2-}$, $L_{3-}$ are $0$ or 
        \item \label{it:rkL = 1 b} each $v_i>0$.
    \end{enumerate}  
\end{prop}

\subsection{Examples}
We now present some examples to showcase the method and its limits. Example~\ref{exmp:SL21(Zp) combinatorial data} shows that some cases of our classification correspond to combinatorial data, which is also shared by uniform groups. Hence a full classification of the uFM property based only on the linking diagram is impossible. Example~\ref{exmp:rk2 and decend} gives a concrete application of case \eqref{it:app rkL = 2 c} of Proposition~\ref{prop:rkL = 2}, where one descends by a power of $p$ before reaching the contradiction. This case of Proposition~\ref{prop:rkL = 2} is indecisive if all $v_i\geq 2$.. Example~\ref{exmp:rk2 and decend} shows that there are examples, where such an assumption is not necessary and it is possible to deduce the uFM property even if $v_p(a_i)\geq 2$ for all $i$.
\begin{exmp}
\label{exmp:SL21(Zp) combinatorial data}
    Consider the group $\SL_2^1(\ZZ_p)$ together with its associated powerful Lie algebra $\mathfrak{sl}^1_2(\ZZ_p)=\{M\in p\cdot \mathfrak{gl}_2(\ZZ_p):\tr(M)=0\}$. We denote
    \begin{align*}
        x\coloneq\exp\left(\begin{smallmatrix}
            0&p\\
            0&0
        \end{smallmatrix}\right),
        \quad 
        y\coloneq\exp\left(\begin{smallmatrix}
            0&0\\
            p&0
        \end{smallmatrix}\right)&,
        \quad 
        z\coloneq\exp\left(\begin{smallmatrix}
            p&p\\
            -p&-p
        \end{smallmatrix}\right),
    \end{align*}
    These are minimal generators of $\SL_2^1(\ZZ_p)$, as can be seen by their reduction modulo $p^2$. Furthermore, one has 
    \begin{align*}
        x^{-2p}[x,z][x,y],\quad y^{2p}[y,z][y,x]^{-1},\quad z^{2p}[z,x][z,y]\in P_3^*(G)
    \end{align*}
    Thus the matrices $A$ and $L$ from~\eqref{eq:combinatorial data} are 
    \begin{align*}
        A=\medmatrix{
            -2&0&0\\
            0&2&0\\
            0&0&2
        }
        \quad \text{and}\quad L=\medmatrix{
            0&-1&1\\
            1&0&1\\
            -1&1&0
        }
    \end{align*}
    It follows that $A_p^{-1}\cdot L$ is symmetric (independently of $p$) and thus Proposition~\ref{prop:rkL = 3} \eqref{it:rkL = 3 a} satisfied.

    For $p=3$, the linking diagram associated with the primes $S=\{7,31,229\}$ with primitive roots $3$, $13$ and $98$ is equivalent to the linking diagram for $\SL_2^1(\ZZ_p)$. This shows, that the linking diagram alone cannot be sufficient to determine the uFM property in all cases.
\end{exmp}
\begin{exmp}\label{exmp:rk2 and decend}
    Let $p=3$ and $S=\{19,73,127\}$, then we have $a_1=6$, $a_2=24$, $a_3=42$ and with respect to the primitive roots $10$, $59$ and $110$ one computes
    \begin{align*}
        L=\medmatrix{
            0&0&1\\
            0&0&-1\\
            1&-1&0
        }
    \end{align*}
    which has rank $2$. Since each $a_i$ is divisible by $3$, we check the conditions of Proposition~\ref{prop:rkL = 2} (\ref{it:rkL = 2 c}). One can choose $m=(1,-1,0)$. Then the matrix $M$ is
    \begin{align*}
        M=\medmatrix{
            0&14\cdot 0&-8\cdot (-1)\\
            -14\cdot 0&0&2\cdot 1\\
            8\cdot (-1)&-2\cdot 1& 0
        }
        =\medmatrix{
            0&0&-1\\
            0&0&-1\\
            1&1&0
        }
    \end{align*}
    The block matrix $\widehat{L}$ now has rank $3$ and therefore according to Proposition~\ref{prop:rkL = 2} \ref{it:app rkL = 2 c} we can replace $a_1$ by $2$, $a_2$ by $8$ and $a_3$ by $14$ and see that we are in the case $(v_1,v_2,v_3)=(0,0,0)$, which is impossible. Thus $G_{K,S}$ does not have any uniform quotients. 
\end{exmp}
\begin{exmp}
\label{exmp:rk2 and Lie algebra not abelian}
    Let us consider a pro-$p$ group $G$ that is weakly presented by the following linking diagram with $v_3=\min\{v_1,v_2,v_3\}$
    \begin{center}
        \begin{tikzpicture}
            \ldVertex{x1}{0,0}{$x_1$}{$p^{v_1}$};
            \ldVertex{x2}{2,0}{$x_2$}{$p^{v_2}$};
            \ldVertex{x3}{4,0}{$x_3$}{$p^{v_3}$};
            \ldEdge{x1}{x2}{$1$}{};
            \ldEdge{x2}{x3}{$1$}{};
        \end{tikzpicture}
    \end{center}
    We also denote the corresponding generators of $G$ by $x_1,x_2,x_3$.
    
    Assume $G$ has a non-trivial uniform quotient $U$ with Lie algebra $\Lie$. Then $U$ has necessarily rank $3$. Let $e_i$ be as usual the basis of $\Lie$ corresponding to the generators $x_i$. Then from \eqref{eq:Bockstein 1 relation} we have in $\fLie$:
    \begin{align*}
        p^{v_1}e_1\equiv \dbb{e_1,e_2},\quad p^{v_2}e_2\equiv \dbb{e_2,e_3},\quad  p^{v_3}e_3\equiv 0\pmod p.
    \end{align*}
    Certainly if $v_3=0$, then this yields a contradiction as then $e_3\in p\fLie$ and hence $d(U)\leq 2$. Thus  $v_3\geq 1$. We deduce from Corollary~\ref{cor:Mod p reduction of FAb LieAlgebras dim3} that either $\dbb{e_1,e_3}$ is trivial or up to scaling $e_2$ modulo $p$. In the first case, one concludes that $\mathfrak{b}_1^1$ is trivial and one can just reduce each $v_i$ by $1$. In the second case, we infer the following congruences from Proposition~\ref{prop:Bockstein 2 relation} 
    \begin{align*}
        p^{v_1}e_1&\equiv \dbb{e_1,e_2}\pmod{p^2,pe_2}\\
        p^{v_2}e_2&\equiv \dbb{e_2,e_3}\pmod{p^2,pe_2}\\
        p^{v_3}e_3&\equiv 0\pmod{p^2,pe_2}
    \end{align*}
    Again if $v_3=1$, then this yields a contradiction, as then $pe_3\equiv \lambda pe_2\pmod {p^2}$. Thus $v_3\geq 2$ and $\dbb{e_1,e_2}\equiv a\cdot pe_2\equiv a\cdot p \cdot \dbb{e_1,e_3}\pmod{p^2}$. This shows by Proposition~\ref{prop:higherBockstein Lie algebra} that $\mathfrak{b}^1_2$ is trivial. Thus one can lift these exact relations modulo $p^3$ and by induction until $p^{v_3}$, where the last congruence gives the desired contradiction.
\end{exmp}
\begin{rem}
    We have not been able to characterize when a similar behavior as in Example~\ref{exmp:rk2 and Lie algebra not abelian} occurs in general, just in terms of the linear algebraic data $A$ and $L$. It would be nice to incorporate this case also into Proposition~\ref{prop:rkL = 2}, since it can only deal with primes for which at least one $q_i$ satisfies $v_p(q_i-1)\leq 2$.
\end{rem}
\subsection{Infinitude of some \texorpdfstring{$G_{K,S}$}{GKS} with \texorpdfstring{$|S|=3$}{|S|=3}}
\label{ssec:Infinitue of GS}
To the authors' knowledge the only known method to show the infinitude of the groups $G_{K,S}$ is by applying the Golod--\v{S}afarevi\v{c} inequality in one way or another,  (see, e.g., \cites{GolodSafaravic1964,VenkovKoch1978,HajirMaire2002,AhlqvistCarlson2025}). 

Since for $|S|=3$ the group $G_{K,S}$ has three generators and at least three relations the Golod--\v{S}afarevi\v{c} theorem can only be applied if these relations are deep in the Zassenhaus filtration (see for example \cite{VenkovKoch1978}). 

We are able to show the infinitude of several $G_{K,S}$ with $|S|=3$ by considering subgroups of $G_{K,S}$ instead of $G_{K,S}$ itself. We demonstrate the strategy using the following example:
\begin{exmp}
\label{exmp:Infinite G_S with S=3}
    Let $p=3$, $K=\QQ$ and $S=\{19,229,571\}$. Then by Lemma~\ref{lem:rank and valuation relation} one can easily show that $G_{\QQ,S}$ has no nontrivial uniform quotients. 

    To show that $G_{\QQ,S}$ is infinite consider the field $K/\QQ$ given by the unique degree $9$ subextension of $\QQ(\zeta_{19})/\QQ$. Since $299\equiv 571\equiv 1\pmod {19}$ we see that $299$ and $571$ are completely split in $K/\QQ$. Hence there are in total $19$ places above $S$ in $K$, which we denote by $\overline{S}$. Consider the group $G_{K,\overline{S}}$, which is a closed subgroup of $G_{\QQ,S}$. By Corollary~\ref{cor:generator and relation rank of GS} we have that
    \begin{align*}
        d(G_{K,\overline{S}})&=1+|\overline{S}|+\dim_{\FF_p}(\B_{\overline{S}}(K))-(r_1+r_2)=11+\dim_{\FF_p}(\B_{\overline{S}}(K))\\
        r(G_{K,\overline{S}})&\leq |\overline{S}|+\dim_{\FF_p}(\B_{\overline{S}}(K))=19+\dim_{\FF_p}(\B_{\overline{S}}(K))
    \end{align*}
    In fact, by explicit computation one can show that $\B_{\overline{S}}(K)=0$, but even without this knowledge the Golod--\v{S}afarevi\v{c} inequality can be applied to $G_{K,\overline{S}}$, showing that $G_{K,\overline{S}}$ is infinite. Thus also $G_{\QQ,S}$ is infinite.
\end{exmp}
\begin{rem}
\label{rem:why our method is need for infinity}
    Let $K=\QQ$ and $S=\{q_1,q_2,q_3\}$ be a triple of tame primes such that $p^2\nmid q_i-1$. According to \cite{Luo2024}*{Thm. 7.13} if the group $G_{\QQ,S}$ is powerful, which is equivalent to $L$ from \eqref{eq:relations in three dim Lie algebra} being invertible by \cite{SymondsWeigel2000}*{Thm. 5.1.6}, then it is either finite or isomorphic to $\SL_2^1(\ZZ_p)$. Therefore to find infinite examples with $|S|=3$, that satisfy the uFM property, one either has to admit primes $q_i$ for which $p^2\mid q_i-1$ or consider matrices $L$ of lower rank. 
\end{rem}
\begin{thm}
\label{thm:Infinite and uFM GS}
    Given a number field $K$ and an odd prime $p$ unramified in $K/\QQ$ with $h_p(K)=1$. Then there exist infinitely many triples of tame places $S=\{v_1,v_2,v_3\}$ such that $G_{K,S}$ satisfies the uFM property and $G_{K,S}$ is infinite. 
\end{thm}
\begin{proof}
    For a number field $M$ we write $r_M\coloneq r_1+r_2$, where $(r_1,r_2)$ is the signature of $M$. Choose a positive integer $k$ such that $p^k\geq r_K+1$.

    We first construct $v_1$. Let $K_k\coloneq K(\zeta_{p^k})$. Then we choose $v_1$ by the \v{C}ebotarev Density Theorem such that the Frobenius
    \begin{align*}
        \bigg(v_1,K_k\Big(\sqrt[p^k]{\mathcal{O}_K^\times V_\emptyset^{\smash{p^{k-1}}}}\Big)/K\bigg)
    \end{align*}
    is trivial. Then by \cite{Gras2003}*{Cor. 2.9.1} there exists a cyclic extension $L/K$ of degree $p^k$ ramified exactly at $v_1$. Note that by construction $v_1$ is completely split in $K(\zeta_{p^k})/K$ and thus $p^k\mid N(v_1)-1$. Furthermore, $(v_1,{\rm Gov}_\emptyset/K)=1$ and thus $\B_{\{v_1\}}(K)=\B_\emptyset(K)$.

    By slightly varying Lemma~\ref{lem:construction of extension with prescribed ramification} one can construct two tame places $v_2$ and $v_3$ satisfying the following conditions for $S=\{v_1,v_2,v_3\}$
    \begin{enumerate}
        \item $v_p(N(v_2)-1)=1$,
        \item $\ell(v_i,v_j)=0$ for each $i,j$,
        \item $\B_{\{v_1,v_2,v_3\}}(K)=\B_\emptyset(K)$,
        \item $v_2$ and $v_3$ are completely split in $L/K$,
        \item $\B_{\overline{S}}(L)=\B_{\{w_1\}}(L)$, where $w_1$ is the unique place in $L$ above $v_1$.
    \end{enumerate}
    Thus all linking numbers are zero, but there is at least one place in $S$ (namely $v_2$), for which $v_p(N(v_2)-1)=1$. Thus by Lemma~\ref{lem:rank and valuation relation} we conclude that $G_{K,S}$ has the uFM property. We notice that
    \begin{align*}
        \dim_{\FF_p}(\B_{\bar{S}}(L))&=\dim_{\FF_p}(\B_{\{v_1\}}(L))=\dim_{\FF_p}(\B_\emptyset(L))-\delta \overset{(*)}=r_L-1-\delta\\&=p^kr_K-1-\delta
    \end{align*}
    where $\delta$ is either $0$ or $1$. The equality ($*$) follows from $h_p(L)=1$, since $L/K$ is totally ramified at exactly $v_1$. The group $G_{L,\bar{S}}$ is an open subgroup of $G_{K,S}$, for which we verify the Golod--\v{S}afarevi\v{c} inequality. By Corollary~\ref{cor:generator and relation rank of GS} we have 
    \begin{align*}
        d(G_{L,\bar S})&=1+|\bar S|+\dim(\B_{\bar S}(L))-r_L=(1+2p^k)-\delta\quad \text{and}\\
        r(G_{L,\bar S})&\leq |\bar S|+  \dim(\B_{\bar S}(L))=1+2p^k+p^kr_K-1-\delta=(2+r_K)p^k-\delta .
    \end{align*}
    We observe that
    \begin{align*}
        d(G_{L,\bar S})^2/4&\geq \tfrac{1}{4}-\delta+p^k(1+p^k)>(2+r_K)p^k-\delta\geq r(G_{L,\bar S}).
    \end{align*}
    Thus, $G_{L,\bar S}$ satisfies the Golod-\v{S}afarevi\v{c} inequality and is therefore non-analytic and hence infinite. Therefore, we can conclude the same for $G_{K,S}$.
\end{proof}
\begin{rems}    
Let $K$ be $\QQ$ or imaginary quadratic and $|S|=3$. Thus $d(G_{K,S})=r(G_{K,S})=3$.
\begin{enumerate}
    \item {\color{black} Following the construction of the proof of Theorem~\ref{thm:Infinite and uFM GS} it} is not necessary to pass to the group $G_{L,S}$, as there are exactly three relations and these are of depth at least $3$ in the Zassenhaus filtration. Thus a refined version of the Golod--\v{S}afarevi\v{c} inequality (see \cite{Koch2002}*{Thm. 7.20}) already implies the infinitude of $G_{K,S}$. 
    \item In this situation it is also possible to construct cases, in which the matrix $L$ has rank up to $2$ using the same strategy as in the proof of Theorem~\ref{thm:Infinite and uFM GS} and following Example~\ref{exmp:Infinite G_S with S=3}. 
\end{enumerate}
\end{rems}
\section{Introducing additional splitting}
\label{sec:Additional splitting}
The idea of considering the uFM property for $G_{K,S}^T$ instead of $G_{K,S}$ is not new. For example it has been applied by F.~Hajir and C.~Maire in \cite{HajirMaire2022} to extend a method developed by N.~Boston in \cites{Boston1992,Boston1999}. The ideas for this section originate from joint work with O.~Hamza and D.~Lim (see \cite{FeuerpfeilHamzaLim2026}). The author would like to thank them for the inspiration and valuable inputs on that problem.
\subsection{General splitting}
\label{ssec:general splitting}
We first start with a simple group-theoretic lemma.
\begin{lem}\label{lem:uFM for FAB after quotient}
    Let~$G$ be a pro-$p$ group with finite abelianization and $n\geq 2$. Fix a minimal generating set~$x_1,...,x_d$ and set $n_\omega$ to be the maximum of $n$ and $\log_p([G:P_n(G)])$. Let~$\mathcal{T}\coloneq \{\tau_1,\dots , \tau_{d-2}\}\subseteq G$ be any family such that for $j=1,...,d-2$
    \begin{align*}
        \tau_j\equiv x_j^{p^{n_{\omega}}} \pmod{P_{n_\omega+2}(G)}.
    \end{align*}
    Denote by $G^\mathcal{T}$ the quotient of $G$ by the normal subgroup generated by $\mathcal{T}$. Then $G^{\mathcal{T}}$ has the uFM property and~$G/P_n(G)\cong G^\mathcal{T}/P_n(G^\mathcal{T})$. 

    Furthermore, if $4r(G)<d(G)^2$, then for $n$ sufficiently large and any family $\mathcal{T}$ as above, $G^\mathcal{T}$ is non-analytic and hence infinite.
\end{lem}
\begin{proof}
    Since the order of $\overline{x}_j$ in $G/P_n(G)$ divides $p^{n_\omega}$ it follows that each $\tau_j\in P_n(G)$ for each $j=1,...,d-2$, implying $\mathcal{T}\subseteq P_n(G)$ and therefore 
    \begin{align*}
        G/P_n(G)\cong G^\mathcal{T}/P_n(G^\mathcal{T}).
    \end{align*}
    To show the uFM property, assume that $U$ is a nontrivial uniform quotient of $G^\mathcal{T}$ and denote by $\pi$ the composition of $G\to G^{\mathcal{T}}\to U$. Since $G$ has finite abelianization, so does $U$. By Lemma~\ref{prop:FAB uniform groups}, we deduce $d(U)\geq 3$. For each $j=1,...,d-2$
    \begin{align*}
        \pi(x_j)^{p^{n_\omega}}=\pi\big(x_j^{p^{n_\omega}}\big)=\pi\big(\tau_j^{-1}x_j^{p^{n_\omega}}\big)\in \pi(P_{n_\omega+2}(G))\subseteq P_{n_\omega+2}(U)=U^{p^{n_\omega+1}}.
    \end{align*}
    The uniformity of $U$ implies $\pi(x_i)\in U^p=\Phi(U)$ for $i=1,...,d-2$, and thus $d(U)\leq 2$. A contradiction.

    The last statement follows from refined version of the Golod--\v{S}afarevi\v{c} theorem together with the fact that $P_n(G)$ is always contained in the $n^{\rm th}$ term of the Zassenhaus filtration. 
\end{proof}
We will apply this lemma to $G=G_{K,S}$, where we will identify $\mathcal{T}$ with a set of Frobenius elements of a set of places $T$ and hence $G^\mathcal{T}$ can be seen as the group $G_{K,S}^T$. To choose $T$ we use the \v{C}ebotarev density theorem and thus we only have control over the Frobenii at $T$ in a finite extension. This is why it was crucial to allow any set $\mathcal{T}$ satisfying the conditions in Lemma~\ref{lem:uFM for FAB after quotient}. 
\begin{thm}
\label{thm:additonal d-2 splitting}
    Given a number field $K$, a set of tame places $S$, and a finite extension $L/K$. Let $d\coloneq d(G_{K,S})$ Then there exist $d-2$ places $T$ such that the group $G_{K,S}^T$ has the uFM property and $T$ is completely split in $L/K$.

    Let $K=\QQ$ or $K$ be imaginary quadratic with $h_p(K)=1$. If $|S|\geq 5$, then $T$ can be chosen such that $G_{K,S}^T$ is non-analytic.
\end{thm}
\begin{proof}
    Without loss of generality, we can assume that $L/K$ is a Galois extension. Otherwise, we pass to the Galois closure of $L$. Let $L'\coloneq K_S(p)\cap L$ and choose $n\geq 3$ such that $P_n(G_{K,S})\subseteq \Gal(K_S(p)/L')$. Set $n_\omega$ as in Lemma~\ref{lem:uFM for FAB after quotient}. Let $K_{n_\omega+2}$ be the fixed field of $P_{n_\omega+2}(G_{K,S})$ and denote by $x_1,...,x_d$ a set of minimal generators of $G_{K,S}$. 

    Then view $x_j^{p^{n_\omega}}$ as elements of $\Gal(K_{n_\omega+2}/K)$. Note that by construction they are elements of $\Gal(K_{n_\omega+2}/L')$. Using the \v{C}ebotarev density theorem, we choose places $w_1,...,w_{d-2}$ of $L'K_{n_\omega+2}$ not above $S$ such that for
    \begin{align*}
        \sigma_j\coloneq (1,x_j^{p^{n_\omega}})\in \Gal(LK_{n_\omega+2}/L')\cong \Gal(L/L')\times \Gal(K_{n_\omega+2}/L')
    \end{align*}
    one has $\sigma_j=(w_j,LK_{n_\omega+2}/K)$. Set $T=\{v_1,...,v_{d-2}\}$ where the $v_j$ are the places of $K$ below $w_j$. Then each $v_j$ is split in $L/K$ and that the Frobenius elements of the $w_j$ in $K_S(p)/K$ are a family as required in Lemma~\ref{lem:uFM for FAB after quotient}. This shows the first part.

    Note that by construction, the $d-2$ new relations are contained in the third term of the Zassenhaus filtration. The polynomial $1-dt+dt^2+(d-2)t^3$ has a root in the interval $[0,1]$ for $d\geq 5$ from which the second claim follows.
\end{proof}
{\color{black}
Determining the additional places $T$ explicitly is quite difficult in general, since the (finite) extensions required to just find the elements in the Galois group are already very large, which makes their computation infeasible. Assuming that one has the extension $K_{n_\omega+2}/K$ and its Galois group $G_{n_\omega+2}$ given, then the problem first amounts to finding the right elements in $G_{n_\omega+2}$. Since the ramification in $K_{n_\omega+2}$ is only tame, one can give an easy upper bound on the discriminant of $K_{n_\omega+2}$, which, in turn, by an explicit version of the \v{C}ebotarev density theorem, yields an upper bound for the norms of the elements in $T$, which one has to check.
}
\subsection{Additional splitting for \texorpdfstring{$|S|=4$}{|S|=4}}
\label{ssec:splitting for S=4}
In some situations one can do better than what is guaranteed by Theorem~\ref{thm:additonal d-2 splitting}.
\begin{prop}
\label{prop:S=4 T=1}
    Let $\Lambda=(\Gamma,a,\ell)$ be a linking diagram with four vertices and such that each $v_p(a_v)=0$ for all $v\in V(\Gamma)$ and $G$ be a pro-$p$ weakly presented by $\Lambda$. 

    Then there exists an element $\tau_0\in [G,G]$ such that for each $\tau\equiv \tau_0\pmod {P^*_3(G)}$ the group $G/\langle\!\langle \tau\rangle\!\rangle$ has the uFM property.
\end{prop}
\begin{proof}    
    Denote by $x_1,...,x_4$ the generators of $G$ corresponding to the vertices $v_1,...,v_4$ of $\Gamma$. For each $i=1,...,4$ set $y_i\coloneq \prod_{j=1}^4x_j^{\ell(v_i,v_j)}$. Note that if $\ell(v_i,v_j)=0$ for each pair $(i,j)$, i.e., $\Gamma$ is totally disconnected, then the image of each $x_i$ in a uniform quotient $U$ is contained in $P_2(U)$ and thus $d(U)=0$. A contradiction.

    Thus we can assume that at least one $\ell(v_i,v_j)\neq 0$ and after reordering the vertices that $\ell(v_1,v_j)\neq 0$ for some $j$. Let $\tau$ be arbitrary such that 
    \begin{align*}
        \tau\equiv [x_2,y_2]^{\ell(v_1,v_2)/a_2}[x_3,y_3]^{\ell(v_1,v_3)/a_3}[x_4,y_4]^{\ell(v_1,v_4)/a_4}\pmod{P_3^*(G)}
    \end{align*}
    By the Koch type presentation we also see that $\tau\equiv y_1^p \pmod{P_3^*(G)}$.
    
    Set $G_1\coloneq G/\langle\!\langle \tau\rangle\!\rangle$ and assume that $\pi:G_1\twoheadrightarrow U$ is the projection of $G_1$ onto a nontrivial uniform quotient. Since $U^{\rm ab}$ is finite, we deduce by Proposition~\ref{prop:FAB uniform groups} that $d(U)\geq 3$.

    Let $e_i\coloneq \log(\pi(x_i))$ for $i=1,...,4$ and set $b_i\coloneq \log(\pi(y_i))$. Then we have the relations $a_ie_i\equiv \dbb{e_i,b_i}\pmod p$ for $i=1,...,4$ and additionally coming from $\tau$
    \begin{align*}
       0\equiv \sum_{i=2}^4a_i^{-1}\ell(v_1,v_i)\dbb{e_i,b_i}\equiv \sum_{i=2}^4 \ell(v_1,v_i)e_i\equiv b_1\pmod p.
    \end{align*}
    This immediately shows that $e_1\equiv \dbb{e_1,b_1}\equiv 0\pmod p$ and thus $d(U)\leq 3$. Since at least one $\ell(v_1,v_i)\neq 0$, the relation yields a linear dependence between $e_2,e_3,e_4$ modulo $p$ and thus $d(U)\leq 2$. This is a contradiction.
\end{proof}
In the number-theoretic situation $\tau$ can --- as before --- be identified with an additional place. Thus by means of the \v{C}ebotarev density theorem, we get the following corollary:
\begin{cor}
\label{cor:One Splitting for S=4}
    Let $K$ be a number field with $h_p(K)=1$ and $\zeta_p\not\in K$ and $S=\{v_1,...,v_4\}$ be a set of four tame places such that $p^2\nmid N(\mathfrak{q}_{v_i})-1$ for $i=1,...,4$. Then there exists a place $w$ of $K$ such that $w$ is completely split in $(K_S(p))^{\rm ab}/K$ and $G_{K,S}^{\smash{\{w\}}}$ has the uFM property.
\end{cor}

In special situations one can drop the assumption that $p^2\nmid N(\mathfrak{q}_{v_i})-1$. For $n\in \NN$ define a linking diagram $\Lambda=(\Gamma,a,\ell)$ by
\begin{center}
    \begin{tikzpicture}[]
        \ldVertex{v1}{0,0}{$v_1$}{$p^n$};
        \ldVertex{v2}{2,0}{$v_2$}{$p^n$};
        \ldVertex{v3}{2,1.5}{$v_3$}{$p^n$};
        \ldVertex{v4}{0,1.5}{$v_4$}{$p^n$};
        \ldEdge{v1}{v2}{1}{bend right=.7cm};
        \ldEdge{v2}{v3}{1}{bend right=1cm};
        \ldEdge{v3}{v4}{1}{bend right=.7cm};
        \ldEdge{v4}{v1}{1}{bend right=1cm};
    \end{tikzpicture}
\end{center}
By Theorem~\ref{thm:Linking diagram approximation by arithmetic} for a number field $K$ we find a set of four tame places $S$ such that $G_{K,S}$ is weakly represented by $\Lambda$. Next choose by the \v{C}ebotarev density theorem a place $\mathfrak{t}$ of $K$ such that its Frobenius coincides with $[x_1,x_2][x_3,x_4]$ in $G_{K,S}/P^*_3(G_{K,S})$ and thus following Example~\ref{exmp:four generators five relations} we see that $G_{K,S}^{\smash T}$ does not admit any nontrivial uniform quotients for $T=\{\mathfrak{t}\}$.

We found further such examples where one can find a single additional relation to apply this technique, but have not been able to identify an approach that subsumes all these constructions.
\begin{quest}
    Can the proof of Proposition~\ref{prop:S=4 T=1} be generalized beyond the case $v_p(a_i)=0$ for all $i$?
\end{quest}
\section{Statistical considerations}
\label{sec:numerical experiments}
Since by the theoretical results of the previous sections (especially \ref{ssec:FAb groups with three generators}), the success of our criterion relies on the precise linking numbers. This section examines in two directions how these behave and what consequences this has for the uFM property of the groups $G_{K,S}$. For simplicity we restrict here to the case $K=\QQ$ and denote $G_{\QQ,S}$ by $G_S$, although many of the observations remain valid over arbitrary number fields. Since the criteria developed are computationally effective, we can investigate them experimentally. The relevant source code relies on the computer algebra system OSCAR~\cite{OSCAR} and can be found under \cite{Feuerpfeil2026Code}.
\subsection{Uniform distribution of linking matrices}
\label{ssec:uniformity of linking matrix}
The proof of Theorem~\ref{thm:Linking diagram approximation by arithmetic}, which shows that any possible combination of linking numbers is realizable is an inductive application of Lemma~\ref{lem:construction of extension with prescribed ramification}, for which the set of valid places forms a \v{C}ebotarev class. This suggests that the linking numbers are in some sense uniformly distributed. In the case where $|S|=2$ this follows from a careful and quantitative study of the proof of Lemma~\ref{lem:construction of extension with prescribed ramification}. 

To make this intuition precise we first introduce some notation. For a prime $p$ and a positive integer $d$, we define $\nabla_{p,d}
\coloneq \{A\in \FF_p^{d\times d}:A_{ii}=0\text{ for all }i\}$ and $\mathcal{P}_{p,d}(x)$ the set of $d$-tuples of primes $(q_1,...,q_d)$ with each $q_i\equiv 1 \pmod p$ and $q_1<q_2<...<q_d\leq x$. For $(q_1,...,q_d)$ such a tuple we define
\begin{align*}
    \lambda(q_1,...,q_d)\coloneq \begin{pmatrix}
        0&\ell(q_1,q_2)&\cdots&\ell(q_1,q_d)\\
        \ell(q_2,q_1)&0&&\ell(q_2,q_d)\\
        \vdots&&\ddots&\vdots\\
        \ell(q_d,q_1)&\ell(q_d,q_2)&\cdots &0
    \end{pmatrix}\in \nabla_{p,d}.
\end{align*}
Note that this is technically not well defined, as the linking numbers depend on a choice of primitive roots, but this does not affect our results. 

We make the following heuristic based on the above observations.
\begin{heu}
\label{heu:matrices uniformly distributed}
    The matrices $\lambda(q_1,...,q_d)\in \nabla_{p,d}$ are asymptotically equidistributed in $\nabla_{p,d}$ as $q_i\to \infty$. In other words, for any matrix $M\in \nabla_{p,d}$
    \begin{align*}
        \frac{\left|\{(q_1,...,q_d)\in \mathcal{P}_{p,d}(x):\lambda(q_1,...,q_d)=M\}\right|}{|\mathcal{P}_{p,d}(x)|}\longrightarrow \frac{1}{p^{d(d-1)}}\qquad \text{as }x\to \infty.
    \end{align*}
\end{heu}
Since it is not feasible to check Heuristic~\ref{heu:matrices uniformly distributed} numerically for every matrix in $\nabla_{p,d}$, we reduce the data. To that end, we choose a finite set $\mathcal{A}\coloneq \{\alpha_1,...,\alpha_m\}$ of non-zero linear forms $\alpha_j:\nabla_{p,d}\to \FF_p$ and define
\begin{align*}
    D_{p,d,\mathcal{A}}(x)\coloneq \max_{\alpha \in \mathcal{A}}\frac{1}{|\mathcal{P}_{p,d}(x)|}\Big|{\sum}_{\mathbf{q}\in \mathcal{P}_{p,d}(x)}\zeta_p^{\alpha(\lambda(\mathbf{q}))}\Big|.
\end{align*}

To make predictions about what would follow from an honest equidistribution, we start with a lemma from probability theory. It uses the Rayleigh distribution $R(b)$ for a parameter $b>0$, which is given by the density function $x/b^2\exp(-x^2/(2b^2))$.
\begin{lem}
\label{lem:max of Rayleigh}
    Let $X_1,...,X_m$ be $m$ independent identically distributed random variables with $X_i\sim R(b)$ where $R(b)$ denotes the Rayleigh distribution for parameter $b$, then 
    \begin{align*}
        \mathbb{E}[\max_{i=1}^m X_i]\leq b(\sqrt{2\log m}+\sqrt{2\pi})
    \end{align*}
\end{lem}
\begin{proof}
    The proof is the standard exponential-moment argument used to estimate maxima of independent random variables (cf. \cite{Wainwright2019}*{Ex. 2.12}), together with the explicit moment generating function of the Rayleigh distribution given by
    \begin{align*}
        \mathbb{E}[\exp(tR(b))]=\int_{0}^\infty e^{btx}x\exp(\tfrac{-x^2}{2}){\rm d}x=1+\sqrt{2\pi}bt\exp(\tfrac{t^2b^2}{2})\Phi(bt)
    \end{align*}
    where $\Phi$ is the standard normal distribution function.
\end{proof}
\begin{prop}
\label{prop:Expected value a_max}
    Fix $r\geq 3$ and an integer $m\geq 1$. For each $1\leq j\leq m$ let $X_{j,1}$,...,$X_{j,n}$ be independent uniform variables with values in $\mu_r$. Set 
    \begin{align*}
        D\coloneq \max_{j=1}^m \tfrac{1}{n}|X_{j,1}+...+X_{j,n}|.
    \end{align*}
    Then $\limsup_{n\to \infty}\sqrt{n}\mathbb{E}[D]\leq   \sqrt{\log m}+\sqrt{\pi}$.
\end{prop}
\begin{proof}
    For now fix $j$. Note that for each $i$ we have $\mathbb{E}[X_{j,i}]=0$ and $\operatorname{Var}(\operatorname{Re} X_{j,i})=\operatorname{Var}(\operatorname{Im} X_{j,i})=\frac{1}{2}$ and $\operatorname{Cov}(\operatorname{Re} X_{j,i},\operatorname{Im} X_{j,i})=0$. Denote by $\bar{X}_{j,n}=(X_{j,1}+....+X_{j,n})/n$. Note that $\mathbb{E}[\bar{X}_{j,n}]=0$.

    Then the multidimensional central limit theorem (cf., e.g., \cite{VanDerVaart2000}*{Ex. 2.1}) yields
    \begin{align*}
        \sqrt{n}\bar{X}_{j,n}\overset{d}{\longrightarrow} \mathcal{N}(0,\tfrac{1}{2}I_2)
    \end{align*}
    with $\mathcal{N}(0,\tfrac{1}{2}I_2)$ being the multivariate normal distribution with respect to the matrix $\tfrac{1}{2}I_2$ centered at $0$ and the limit is taken in the distributional sense. 

    The norm in $\CC=\RR^2$ transforms the normal distribution $\mathcal{N}(0,\frac{1}{2}I_2)$ into a Rayleigh distribution $R$ with parameter $1/\sqrt{2}$. 

    Since each the distributions $|\sqrt{n}X_{j,n}|$ are independent and identically distributed, and converge to $R$ we deduce from Lemma~\ref{lem:max of Rayleigh} in combination with the Continuous Mapping Theorem \cite{VanDerVaart2000}*{Thm. 2.3}
    \begin{align*}
        \limsup_{n\to \infty}\sqrt{n}\mathbb{E}[D]\leq   \sqrt{\log m}+\sqrt{\pi}.
    \end{align*}
\end{proof}
The following proposition is an easy consequence of Proposition~\ref{prop:Expected value a_max}.
\begin{prop}
\label{prop:Consequence of Heuristic}
    Under Heuristic~\ref{heu:matrices uniformly distributed} for any finite set $\mathcal{A}$ of nontrivial linear forms of $\nabla_{p,d}$ one has
    \begin{align*}
        \limsup_{x\to \infty}D_{p,d,\mathcal{A}}(x)\cdot |\mathcal{P}_{p,d}(x)|^{1/2}\leq \sqrt{\log |\mathcal{A}|}+\sqrt{\pi}.
    \end{align*}
\end{prop}
Figure~\ref{fig:Uniformity of linking matrices experiment} presents data from numerical experiments. For practical reasons, the set $\mathcal{P}_{p,d}(x)$ is given as a list of tuples using the lexicographical order on the $d$-tuples of primes indexed by a natural number $N$. The plot shows the function $D_{p,d,\mathcal{A}}(N)$, which coincides with the above defined function $D_{p,d,\mathcal{A}}(x)$, but we sum over the first $N$ tuples in the list. 

The horizontal axis is scaled logarithmically and shows $N$. The dotted black line is given by $(\sqrt{\log|\mathcal{A}|}+\sqrt{\pi})/\sqrt{N}$, which is the upper bound from Proposition~\ref{prop:Consequence of Heuristic} for the predicted expectation and the black line is given by $\sqrt{\log |\mathcal{A}|}/\sqrt{N}$, which seems to approximate the curves best. 

In all experiments presented here we have $m=|\mathcal{A}|=30$ and the linear forms have been chosen randomly in $\nabla_{p,d}^\vee\setminus \{0\}$. We conducted further experiments with other values of $m$ and did not notice any dependence. 
\begin{figure}[h]
    \centering
    \includegraphics[width=0.9\linewidth]{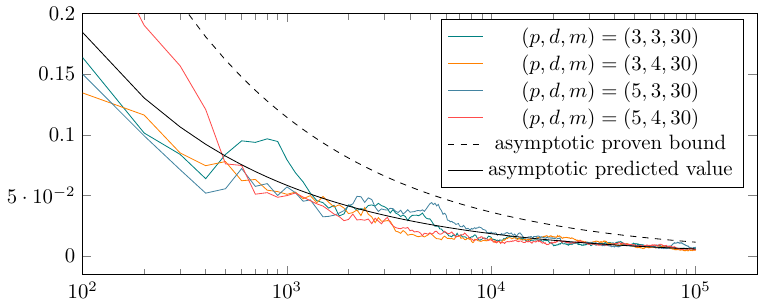}
    \caption{Empirical values of $D_{p,d,\mathcal{A}}(N)$ together with the predicted expectation $\sqrt{\log m}/\sqrt{N}$ and the upper bound from Proposition~\ref{prop:Consequence of Heuristic}}
    \label{fig:Uniformity of linking matrices experiment}
\end{figure}
In all tested cases, the observed values remain well below the theoretical upper bound and fluctuate around the predicted expectation. No systematic dependence on $p$ or $d$ is visible within the tested range, supporting the claim of Heuristic~\ref{heu:matrices uniformly distributed}.
\subsection{Validity of the uFM property for \texorpdfstring{$G_{K,S}$}{GKS} based on the technique in Section~\ref{ssec:FAb groups with three generators}}
\label{ssec:Statistics on uFM property}
We investigate how often the criterion of Section~\ref{ssec:FAb groups with three generators} succeeds in verifying the uFM property. Heuristic~\ref{heu:matrices uniformly distributed} suggests that the case where $\rk(L)=3$ (for $L$ constructed as in \eqref{eq:combinatorial data}) should happen in the overwhelming majority of cases  --- especially for large $p$. In this situation we have the strongest restrictions. Furthermore, we expect that the conditions given by Proposition~\ref{prop:rkL = 3} \ref{it:rkL = 3 a} are rarely satisfied, as the product of a random (invertible) matrix in $\nabla_{3,p}$ with a random (invertible) diagonal matrix is almost never symmetric. Thus Heuristic~\ref{heu:matrices uniformly distributed} implies the following corollary.
\begin{cor}
\label{cor:uFM always satisfied}
    Assume that Heuristic~\ref{heu:matrices uniformly distributed} is true. Then, the proportion of triples $S=\{q_1,q_2,q_3\}$ for which the method described in Section~\ref{ssec:FAb groups with three generators} yields the validity of the uFM property tends to $1$ as $p\to \infty$.
\end{cor}
Tables~\ref{tab:uFM p=3} to \ref{tab:uFM p=7} show the results of numerical experiments for $p=3,5$ and $7$. In the column labeled \emph{total} we denote the number $|\mathcal{P}_{3,p}(x)|$ and in the subsequent columns the number of such triples $(q_1,q_2,q_3)$ for which the matrix $L$ constructed from the linking numbers as in \eqref{eq:combinatorial data}. Below we indicate the number of cases, where we can verify that $G_S$ has the uFM property.
\begin{table}[h]
    \centering
    \begin{tabular}{c||c|c|c|c|c}
         & total & $\rk L=3$ & $\rk L=2$ & $\rk L=1$ & $\rk L=0$\\
         \hline \hline
        total & $1456935$ & $688692$ & $701321$ & $65342$ & $1580$\\
        \hline
        \multirow{2}{*}{verified} & $1409890$ & $683761$ & $665420$ & $59216$ & $1493$\\
        & \scriptsize ($96.8\%$)& \scriptsize ($99.3\%$) & \scriptsize ($94.9\%$)&\scriptsize($90.6\%$) & \scriptsize($94.5\%$)
    \end{tabular}
    \caption{Verification rates of the criterion from Section~\ref{ssec:FAb groups with three generators} for $p=3$ and $x=3000$.}
    \label{tab:uFM p=3}
\end{table}
\begin{table}[h]
    \centering
    \begin{tabular}{c||c|c|c|c|c}
         & total & $\rk L=3$ & $\rk L=2$ & $\rk L=1$ & $\rk L=0$\\
         \hline \hline
        total & $708561$ & $496457$ & $206942$ & $5112$ & $50$\\
        \hline
        \multirow{2}{*}{verified} & $704380$ & $494952$ & $204376$ & $5002$ & $50$\\
        & \scriptsize ($99.4\%$)& \scriptsize ($99.7\%$) & \scriptsize ($98.8\%$)&\scriptsize($97.8\%$) & \scriptsize($100.0\%$)
    \end{tabular}
    \caption{Verification rates of the criterion from Section~\ref{ssec:FAb groups with three generators} for $p=5$ and $x=5000$.}
    \label{tab:uFM p=5}
\end{table}
\begin{table}[h]
    \centering
    \begin{tabular}{c||c|c|c|c|c}
         & total & $\rk L=3$ & $\rk L=2$ & $\rk L=1$ & $\rk L=0$\\
         \hline \hline
        total & $508080$ & $407374$ & $99813$ & $889$ & $4$\\
        \hline
        \multirow{2}{*}{verified} & $507002$ & $406832$ & $99284$ & $882$ & $4$\\
        & \scriptsize ($99.8\%$)& \scriptsize ($99.9\%$) & \scriptsize ($99.5\%$)&\scriptsize($99.2\%$) & \scriptsize($100.0\%$)
    \end{tabular}
    \caption{Verification rates of the criterion from Section~\ref{ssec:FAb groups with three generators} for $p=7$ and $x=7000$.}
    \label{tab:uFM p=7}
\end{table}
In the experiments shown here, the verification rate exceeds $96\%$ and increases from $96.8\%$ for $p=3$ to $99.8\%$ for $p=7$. This strongly suggests that the criterion of Section~\ref{ssec:FAb groups with three generators} captures the vast majority of tame Fontaine–-Mazur groups arising from triples of primes.

Taken together, the experiments provide evidence for two complementary phenomena. First, linking matrices appear to behave as random matrices over $\FF_p$ supporting Heuristic \ref{heu:matrices uniformly distributed}. Second, under this heuristic the criterion developed in Section~\ref{ssec:FAb groups with three generators} verifies the uFM property for an overwhelming proportion of examples of triples of tame primes, as implied by Corollary~\ref{cor:uFM always satisfied}. 

%% file: appendix.tex
\section{Proofs of Propositions~\ref{prop:rkL = 3}, \ref{prop:rkL = 2}, and \ref{prop:rkL = 1}}
\label{sec:Appendix with proofs}
In this appendix we give the proofs to verify that the method presented in Section~\ref{ssec:FAb groups with three generators} is indeed correct. 
\begin{repprop}{prop:rkL = 3}
    If $\rk L=3$ and there exists a nontrivial uniform quotient of $G$, then either
    \begin{enumerate}[(a)]
        \item \label{it:app rkL = 3 a} each $v_i$ is $0$ and $A_p^{-1}L$ is symmetric or
        \item \label{it:app rkL = 3 b} each $v_i> 0$ and there is a three-dimensional FAb $\ZZ_p$-Lie algebra $\widetilde{\Lie}$ such that $\widetilde{\Lie}_p$ satisfies the relations (\ref{eq:relations in three dim Lie algebra}) with $a_i$ replaced by $a_i/p^{v_1}$.
    \end{enumerate}  
\end{repprop}
\begin{proof}
    We first show, that in this situation it cannot happen that there are some $v_i$ that are $0$ and others that are non-zero. We assume that $v_1\leq v_2\leq v_3$. If $v_1=0$, but $v_2\neq 0$, then $\Lie^\flat_p=\langle e_1\rangle$ and hence by Corollary~\ref{cor:Mod p reduction of FAb LieAlgebras dim3} $e_1$ is contained in the center of $\Lie^\flat_p$. A contradiction to $e_1\equiv \dbb{e_1,\ell_{12}e_2+\ell_{13}e_3}$. In the case that $v_1=v_2=0$ and $v_3>0$ we see that $\Lie^\flat_p$ is two dimension. Hence by \cite{Jacobson1962}*{I.4 (d)} it follows that $\dbb{e_1,e_2}\equiv 0\pmod p$. Hence
    \begin{align*}
        e_1\equiv \ell_{13}f_1\pmod p\quad\text{and}\quad e_2\equiv \ell_{23}f_2\pmod p
    \end{align*}
    Hence $f_1$ and $f_2$ are linearly independent in $\Lie_p^\flat$. Thus the relation $0\equiv \dbb{e_3,\ell_{31}e_1+\ell_{32}e_2}\equiv -\ell_{31}f_1+\ell_{32}f_2$ implies $L_{3-}=0$. A contradiction to the invertibility of $L$.

    If each $v_i=0$, then  $\Lie_p$ is perfect and by \cite{Jacobson1962}*{I.4 (e)} the matrix $B=(B_{ij})\in \GL_3(\FF_p)$ with $f_i=\sum_{i=1}^3B_{ij}\overline{e}_j$ for all $i=1,2,3$ is symmetric by the Jacobi identity in $\Lie_p$. It is easy to see that $B=L^{-1}A_p$. As the inverse of a symmetric matrix is symmetric, it follows that $A_p^{-1}L$ is symmetric too.
    
    Now assume that each $v_i>0$. We show that $\dbb{\Lie^\flat,\Lie^\flat}\subseteq p^{r}\Lie^\flat$ for each $r\leq v_1$. For $r=1$ this follows from the fact that $\Lie_p^\flat$ is abelian (since $L$ is invertible, it follows from \eqref{eq:relations in three dim Lie algebra} that $f_i=0$ for each $i$). Now assume that it is known for some $r<v_1$. Then we use Proposition~\ref{prop:First r Bocksteins vanish} to deduce the same relations with $a_j$ replaced by $a_j/p^{r}$. If $r+1<v_1$ the left side of the congruences is still $0$ mod $p$ and hence $\Lie^{r\flat}_p$ still abelian and hence $\dbb{\Lie^\flat,\Lie^\flat}\subseteq p^{r+1}\Lie^\flat$. If $r+1=v_1$ then we recover the claimed relations.
\end{proof}
\begin{repprop}{prop:rkL = 2}
    If $\rk L=2$ and there exists a nontrivial uniform quotient of $G$, then either
    \begin{enumerate}[(a)]
        \item \label{it:app rkL = 2 a} $v_1=v_2=0$, $v_3> 0$ and the rows $L_{1-}$, $L_{2-}$ are linearly independent, $L_{3-}=0$, and both $\ell_{13}$ and $\ell_{23}$ are non-zero.
        \item \label{it:app rkL = 2 b} $v_1=0$, $v_2,v_3> 0$ and $L_{1-}$ is non-zero, $L_{2-}$ and $L_{3-}$ are linearly dependent, and $L_{21}=L_{31}=0$.
        \item \label{it:app rkL = 2 c} $v_1,v_2,v_3> 0$. In this case let $0\neq \mu=(\mu_1,\mu_2,\mu_3)\in \FF_p^3$ such that $\mu L=0$ and set
        \begin{align*}
            M\coloneq\medmatrix{
                0&(a_3/p) \mu_3&-(a_2/p) \mu_2\\
                -(a_3/p) \mu_3&0&(a_1/p) \mu_1\\
                (a_2/p) \mu_2&-(a_1/p) \mu_1&0
            }
            \in \FF_p^{3\times 3}\quad \text{and}\quad \widehat{L}\coloneq \medmatrix{
                L\\
                M
            }
        \end{align*}
        If the matrix $\widehat{L}$ has rank $3$, then there exists a three-dimensional FAb $\ZZ_p$-Lie algebra $\widetilde{\Lie}$ such that $\widetilde{\Lie}_p$ satisfies the relations (\ref{eq:relations in three dim Lie algebra}) with $a_i$ replaced by $a_i/p$.
    \end{enumerate}  
\end{repprop}
\begin{proof}
    By Lemma~\ref{lem:rank and valuation relation} we know that $v_3>0$. In particular we only have to consider the given three cases. 
    \begin{itemize}
        \item[\eqref{it:app rkL = 2 a}] There exist $\lambda_1,\lambda_2$ such that $L_{3-}=\lambda_1L_{1-}+\lambda_2L_{2-}$. Applying \eqref{eq:relations in three dim Lie algebra} yields $\lambda_1a_1e_1+\lambda_2a_2e_2\equiv 0\pmod p$ and thus $\lambda_1=\lambda_2=0$ by the linear independence of $e_1$ and $e_2$. Hence also $L_{3-}=0$.

        Since $e_1,e_2\in (\Lie_p^\flat)'$, $(\Lie_p^\flat)'$ has either dimension $2$ or $3$. In the first case $e_1$ and $e_2$ commute by \cite{Jacobson1962}*{I.4 (d)}. This is only possible if $l_{13}$ and $l_{23}$ are non-zero. Thus we are left with the case, where $\Lie_p^\flat$ is perfect. In that case $e_3\equiv w_1f_1+w_2f_2+w_3f_3\pmod p$ for suitable $w_i\in \FF_p$. Thus we get
        \begin{align*}
            \medmatrix{
            0&-\ell_{13}&\ell_{12}\\
            \ell_{23}&0&-\ell_{21}\\
            w_1&w_2&w_3
        }
        \medmatrix{
            f_1\\ f_2\\ f_3
        }
        =\medmatrix{
            a_1e_1\\
            a_2e_2\\
            e_3
        }
        \end{align*}
        By \cite{Jacobson1962}*{I.4 (e)} the Jacobi identity in $\mathfrak{g}_p^\flat$ implies that
        \begin{align*}
            \medmatrix{
            a_1&0&0\\
            0&a_2&0\\
            0&0&1
        }^{-1}
        \medmatrix{
            0&-\ell_{13}&\ell_{12}\\
            \ell_{23}&0&-\ell_{21}\\
            w_1&w_2&w_3
        }
        =\medmatrix{
            0&-\ell_{13}a_1^{-1}&\ell_{12}a_1^{-1}\\
            \ell_{23}a_2^{-1}&0&-\ell_{21}a_2^{-1}\\
            w_1&w_2&w_3
        }
        \end{align*}
        is symmetric. Hence $-\ell_{13}a_2\equiv \ell_{23}a_1\pmod p$ and this value is non-zero implying that $\ell_{13}$ and $\ell_{23}$ are both non-zero, as this matrix is invertible by construction.
        \item[\eqref{it:app rkL = 2 b}] Clearly the row $L_{1-}$ non-zero and by a similar argument as before linearly independent to $L_{2-}$ and $L_{3-}$. Thus, $L_{2-}$ and $L_{3-}$ are linearly dependent.

        If $(\Lie_p^\flat)'$ would have dimension $1$ it would be generated by $e_1$ and hence $e_1$ central modulo $p$ by Corollary~\ref{cor:Mod p reduction of FAb LieAlgebras dim3} \eqref{it:one dimensional derived implies central}. A contradiction to $e_1\equiv \dbb{e_1,\ell_{12}e_2+\ell_{13}e_3}\pmod p$. Thus $(\Lie_p^\flat)'$ has to have dimension $2$. Now pick $0\neq w=w_2e_2+w_3e_3\in (\Lie_p^\flat)'$. Then  $0=\dbb{e_1,w}=\alpha_2 f_3-\alpha_3f_2$ by \cite{Jacobson1962}*{I.4 (d)}. Let
         \begin{align*}
             \lambda=v_1(e_2\wedge e_3)+v_2(e_3\wedge e_1)+v_3(e_1\wedge e_2)\in \Lambda^2(\Lie_p)
         \end{align*}
         such that $\dbb{\lambda}=v_1f_1+v_2f_2+v_3f_3=w$. Now consider the matrix $L'$ defined by
         \begin{align*}
             L'=\medmatrix{
                 L\\
                 v
             }
             \in \FF_p^{4\times 3}
         \end{align*}
         with $v=(v_1,v_2,v_3)$. Then the rank of $L'$ is full and there are $\lambda_1,\lambda_2,\lambda_3,\lambda_4$ such that $(\lambda_1,...,\lambda_4)L'=(0,-\alpha_3,\alpha_2)$. By multiplication from the right with $(f_1,f_2,f_3)^T$ we deduce that $\lambda_1=\lambda_4=0$ and therefore
        \begin{align*}
            \begin{pmatrix}
                \lambda_2,\lambda_3
            \end{pmatrix}
            \begin{pmatrix}
                L_{21}&0&L_{23}\\
                L_{31}&L_{32}&0
            \end{pmatrix}
            =(0,-\alpha_3,\alpha_2).
        \end{align*}
        Since the rows are linearly dependent, so are the columns and thus if the first column is non-zero, then $\alpha_3$ and $\alpha_2$ are $0$. Contradicting, that $w$ is nontrivial.
        \item[\eqref{it:app rkL = 2 c}] If $(\Lie_p^\flat)'=0$, then we can apply the same argument as in the proof of Proposition~\ref{prop:rkL = 3} \ref{it:rkL = 3 b} and descend by a power of $p$. Thus it remains to show, that the assumptions imply that $(\Lie_p^\flat)'=0$. 

        Since $\rk L=2$, the dimension of $(\Lie_p^\flat)'$ can be at most $1$. Assume it is one. Let $w$ be a generator of $(\Lie_p^\flat)'$. Then by Proposition~\ref{prop:FAb Lie algebra dim3} \eqref{it:one dimensional derived implies central} we conclude that 
        \begin{align*}
            \sum_{i=1}^3 \widehat{\mu}_ia_ie_i\equiv 0\pmod{p^2,p(\Lie^\flat)'}
        \end{align*}
        For any lift of the $\mu_i$ to $\ZZ_p$. Thus, $\sum_{i=1}^3 (a_i/p)\cdot \mu_i \cdot \overline{e}_i\in (\Lie_p^\flat)'=\langle w\rangle$. If this element is zero, then $(a_i/p)\mu_i\equiv 0\pmod p$ and thus $M=0$ a contradiction to our assumption. Without loss of generality we can assume $w=\sum_{i=1}^3 (a_i/p)\cdot \mu_i \cdot \overline{e}_i$.
        
        By Corollary~\ref{cor:Mod p reduction of FAb LieAlgebras dim3} \eqref{it:one dimensional derived implies central} $w$ has to be central. Thus, we get three additional equations by $\dbb{e_i,w}=0$ for $i=1,2,3$, which yield exactly the rows of $M$. By assumption the rank of the block matrix $\widehat{L}$ is $3$ and therefore as trivial right-kernel, showing that $f_1=f_2=f_3=0$. A contradiction to the assumption that $(\Lie_p^\flat)'$ has dimension $1$.
    \end{itemize}
\end{proof}
\begin{repprop}{prop:rkL = 1}
    If $\rk L=1$ and there exists a uniform quotient of $G$, then either
    \begin{enumerate}[(a)]
        \item $v_1=0$, $v_2,v_3>0$ and the rows $L_{2-}$, $L_{3-}$ are $0$ or 
        \item  each $v_i>0$.
    \end{enumerate}  
\end{repprop}
\begin{proof}
    This is a direct consequence of Lemma~\ref{lem:rank and valuation relation}.
\end{proof}

%% file: references.bib
@misc{stacks-project,
  author       = {The {Stacks project authors}},
  title        = {The {S}tacks project},
  howpublished = {\url{https://stacks.math.columbia.edu}},
  year         = {2026},
}

@article{HuberKingsNaumann2011,
 author = {Huber, A. and Kings, G. and Naumann, N.},
 title = {Some complements to the {Lazard} isomorphism},
 fjournal = {Compositio Mathematica},
 journal = {Compos. Math.},
 issn = {0010-437X},
 volume = {147},
 number = {1},
 pages = {235--262},
 year = {2011},
 language = {English},
 doi = {10.1112/S0010437X10004884},
 zbMATH = {5854934},
 Zbl = {1268.20051}
}

@book{NSW2000,
 author = {Neukirch, J. and Schmidt, A. and Wingberg, K.},
 title = {Cohomology of number fields},
 edition = {2nd ed.},
 fseries = {Grundlehren der Mathematischen Wissenschaften},
 series = {Grundlehren Math. Wiss.},
 issn = {0072-7830},
 volume = {323},
 isbn = {978-3-540-37888-4},
 year = {2008},
 publisher = {Berlin: Springer},
 language = {English},
 zbMATH = {5162349},
 Zbl = {1136.11001}
}

@book{Jacobson1962,
 author = {Jacobson, N.},
 title = {Lie algebras},
 fseries = {Interscience Tracts in Pure and Applied Mathematics},
 series = {Intersci. Tracts Pure Appl. Math.},
 volume = {10},
 year = {1962},
 publisher = {Interscience Publishers, New York, NY},
 language = {English},
 zbMATH = {3196329},
 Zbl = {0121.27504}
}

@article{Labute2006,
 author = {Labute, J.},
 title = {Mild pro-{$p$}-groups and {Galois} groups of {$p$}-extensions of {$\mathbb{Q}$}},
 fjournal = {Journal f{\"u}r die Reine und Angewandte Mathematik},
 journal = {J. Reine Angew. Math.},
 issn = {0075-4102},
 volume = {596},
 pages = {155--182},
 year = {2006},
 language = {English},
 doi = {10.1515/CRELLE.2006.058},
 zbMATH = {5080511},
 Zbl = {1122.11076}
}

@article{Labute2014,
 author = {Labute, J.},
 title = {Linking numbers and the tame {Fontaine}-{Mazur} conjecture},
 fjournal = {Annales Math{\'e}matiques du Qu{\'e}bec},
 journal = {Ann. Math. Qu{\'e}.},
 issn = {2195-4755},
 volume = {38},
 number = {1},
 pages = {61--71},
 year = {2014},
 language = {English},
 doi = {10.1007/s40316-014-0012-4},
 zbMATH = {6346468},
 Zbl = {1314.11067}
}

@book{McCeary2001,
 author = {McCleary, J.},
 title = {A user's guide to spectral sequences.},
 edition = {2nd ed.},
 fseries = {Cambridge Studies in Advanced Mathematics},
 series = {Camb. Stud. Adv. Math.},
 volume = {58},
 isbn = {0-521-56759-9},
 year = {2001},
 publisher = {Cambridge: Cambridge University Press},
 language = {English},
 zbMATH = {1565334},
 Zbl = {0959.55001}
}

@article{Tate1976,
 author = {Tate, J.},
 title = {Relations between {{\(K_2\)}} and {Galois} cohomology},
 fjournal = {Inventiones Mathematicae},
 journal = {Invent. Math.},
 issn = {0020-9910},
 volume = {36},
 pages = {257--274},
 year = {1976},
 language = {English},
 doi = {10.1007/BF01390012},
 url = {https://eudml.org/doc/142421},
 zbMATH = {3559703},
 Zbl = {0359.12011}
}

@book{DDMS1999,
 author = {Dixon, J. D. and du Sautoy, M. P. F. and Mann, A. and Segal, D.},
 title = {Analytic pro-{{\(p\)}} groups. {Revised} and enlarged by {Marcus} du {Sautoy} and {Dan} {Segal}.},
 edition = {2nd ed.},
 fseries = {Cambridge Studies in Advanced Mathematics},
 series = {Camb. Stud. Adv. Math.},
 volume = {61},
 isbn = {0-521-65011-9},
 year = {1999},
 publisher = {Cambridge: Cambridge University Press},
 language = {English},
 zbMATH = {1339096},
 Zbl = {0934.20001}
}

@misc{MaireSankara2025,
 author = {Maire, C. and Sankara, K.},
 title = {On {S}-{Split} p-{Hilbert} {Class} {Field} {Towers} with {Prescribed} {Galois} {Groups}},
 year = {2025},
 howpublished = {Preprint, {arXiv}:2508.07946 [math.{NT}] (2025)},
 url = {https://arxiv.org/abs/2508.07946},
 arXiv = {arXiv:2508.07946}
}

@book{Koch2002,
 author = {Koch, H.},
 title = {Galois theory of {{\(p\)}}-extensions. {Transl}. from the {German} by {F}. {Lemmermeyer}},
 fseries = {Springer Monographs in Mathematics},
 series = {Springer Monogr. Math.},
 issn = {1439-7382},
 isbn = {3-540-43629-4},
 year = {2002},
 publisher = {Berlin: Springer},
 language = {English},
 zbMATH = {1798938},
 Zbl = {1023.11002}
}

@article{Massey1953,
 author = {Massey, W. S.},
 title = {Exact couples in algebraic topology. {I}.-{V}},
 fjournal = {Annals of Mathematics. Second Series},
 journal = {Ann. Math. (2)},
 issn = {0003-486X},
 volume = {56},
 pages = {363--396},
 year = {1953},
 language = {English},
 doi = {10.2307/1969805},
 zbMATH = {3077484},
 Zbl = {0049.24002}
}

@book{Weibel1994,
 author = {Weibel, C. A.},
 title = {An introduction to homological algebra},
 fseries = {Cambridge Studies in Advanced Mathematics},
 series = {Camb. Stud. Adv. Math.},
 volume = {38},
 isbn = {0-521-43500-5},
 year = {1994},
 publisher = {Cambridge: Cambridge University Press},
 language = {English},
 zbMATH = {595200},
 Zbl = {0797.18001}
}

@article{Lazard1965,
 author = {Lazard, M.},
 title = {{{\(p\)}}-adic analytic groups},
 fjournal = {Publications Math{\'e}matiques},
 journal = {Publ. Math., Inst. Hautes {\'E}tud. Sci.},
 issn = {0073-8301},
 volume = {26},
 pages = {389--603},
 year = {1965},
 language = {French},
 url = {https://eudml.org/doc/103856},
 zbMATH = {3225421},
 Zbl = {0139.02302}
}

@incollection{FontaineMazur1995,
 author = {Fontaine, J-M. and Mazur, B.},
 title = {Geometric {Galois} representations},
 booktitle = {Elliptic curves, modular forms, \& Fermat's last theorem. Proceedings ot the conference on elliptic curves and modular forms held at the Chinese University of Hong Kong, December 18-21, 1993},
 isbn = {1-57146-026-8},
 pages = {41--78},
 year = {1995},
 publisher = {Cambridge, MA: International Press},
 language = {English},
 zbMATH = {824714},
 Zbl = {0839.14011}
}

@article{HajirMaire2022,
 author = {Hajir, F. and Maire, C.},
 title = {Analytic {Lie} extensions of number fields with cyclic fixed points and tame ramification},
 fjournal = {Journal of the Ramanujan Mathematical Society},
 journal = {J. Ramanujan Math. Soc.},
 issn = {0970-1249},
 volume = {37},
 number = {1},
 pages = {63--85},
 year = {2022},
 language = {English},
 url = {www.mathjournals.org/jrms/2022-037-001/2022-037-001-006.html},
 zbMATH = {7500313},
 Zbl = {1491.11057}
}

@article{Maire2018,
 author = {Maire, C.},
 title = {On the quotients of the maximal unramified 2-extension of a number field},
 fjournal = {Documenta Mathematica},
 journal = {Doc. Math.},
 issn = {1431-0635},
 volume = {23},
 pages = {1263--1290},
 year = {2018},
 language = {English},
 doi = {10.25537/dm.2018v23.1263-1290},
 zbMATH = {6957407},
 Zbl = {1477.11190}
}

@book{OMeara1973,
 author = {O'Meara, O. T.},
 title = {Introduction to quadratic forms. 3rd corrected printing},
 fseries = {Grundlehren der Mathematischen Wissenschaften},
 series = {Grundlehren Math. Wiss.},
 issn = {0072-7830},
 volume = {117},
 year = {1973},
 publisher = {Springer, Cham},
 language = {English},
 zbMATH = {3409446},
 Zbl = {0259.10018}
}

@article{Ramakrishna2008,
 author = {Ramakrishna, R.},
 title = {Constructing {Galois} representations with very large image},
 fjournal = {Canadian Journal of Mathematics},
 journal = {Can. J. Math.},
 issn = {0008-414X},
 volume = {60},
 number = {1},
 pages = {208--221},
 year = {2008},
 language = {English},
 doi = {10.4153/CJM-2008-009-7},
 zbMATH = {5237301},
 Zbl = {1197.11062}
}

@misc{Koch1966,
 author = {Koch, H.},
 title = {l-{Erweiterungen} mit vorgegebenen {Verzweigungsstellen}},
 year = {1966},
 language = {German},
 howpublished = {Algebr. {Zahlentheorie}, {Ber}. {Tagung} {Math}. {Forschinst}. {Oberwolfach} 1964, 139-141 (1966).},
 zbMATH = {3317154},
 Zbl = {0199.09801}
}

@book{Gras2003,
 author = {Gras, G.},
 title = {Class field theory. {From} theory to practice},
 fseries = {Springer Monographs in Mathematics},
 series = {Springer Monogr. Math.},
 issn = {1439-7382},
 isbn = {3-540-44133-6},
 year = {2003},
 publisher = {Berlin: Springer},
 language = {English},
 zbMATH = {1834428},
 Zbl = {1019.11032}
}

@article{GrasMunnier1998,
     author = {G. Gras and A. Munnier},
     title = {Extensions cycliques $T$-totalement ramifi\'ees},
     journal = {Publications math\'ematiques de Besan\c{c}on. Alg\`ebre et th\'eorie des nombres},
     eid = {6},
     pages = {1--17},
     publisher = {Presses universitaires de Franche-Comt\'e},
     year = {1998},
     doi = {10.5802/pmb.a-91},
     language = {fr},
     url = {https://pmb.centre-mersenne.org/articles/10.5802/pmb.a-91/}
}

@article{FeuerpfeilHamzaLim2026,
 author = {J. Feuerpfeil and O. Hamza and D. Lim},
 title = {Tame {Galois} {Groups}, {Linking} {Numbers} and {Mildness}},
 year = {2026},
 howpublished = {Preprint, {arXiv}:2606.01083 [math.{NT}] (2026)},
 url = {https://arxiv.org/abs/2606.01083},
 arXiv = {arXiv:2606.01083}
}

@article{BrowderPakianathan2000,
 author = {Browder, W. and Pakianathan, J.},
 title = {Cohomology of uniformly powerful {{\(p\)}}-groups},
 fjournal = {Transactions of the American Mathematical Society},
 journal = {Trans. Am. Math. Soc.},
 issn = {0002-9947},
 volume = {352},
 number = {6},
 pages = {2659--2688},
 year = {2000},
 language = {English},
 doi = {10.1090/S0002-9947-99-02470-8},
 zbMATH = {1449825},
 Zbl = {0987.20027}
}

@article{Boston1992,
 author = {Boston, N.},
 title = {Some cases of the {Fontaine}-{Mazur} conjecture},
 fjournal = {Journal of Number Theory},
 journal = {J. Number Theory},
 issn = {0022-314X},
 volume = {42},
 number = {3},
 pages = {285--291},
 year = {1992},
 language = {English},
 doi = {10.1016/0022-314X(92)90093-5},
 zbMATH = {98391},
 Zbl = {0768.11044}
}

@article{Boston1999,
 author = {Boston, N.},
 title = {Some cases of the {Fontaine}-{Mazur} conjecture. {II}},
 fjournal = {Journal of Number Theory},
 journal = {J. Number Theory},
 issn = {0022-314X},
 volume = {75},
 number = {2},
 pages = {161--169},
 year = {1999},
 language = {English},
 doi = {10.1006/jnth.1998.2337},
 zbMATH = {1275040},
 Zbl = {0928.11050}
}

@article{GonzaleSanchezKlopsch2009,
 author = {Gonz{\'a}lez-S{\'a}nchez, J. and Klopsch, B.},
 title = {Analytic pro-{{\(p\)}} groups of small dimensions},
 fjournal = {Journal of Group Theory},
 journal = {J. Group Theory},
 issn = {1433-5883},
 volume = {12},
 number = {5},
 pages = {711--734},
 year = {2009},
 language = {English},
 doi = {10.1515/JGT.2009.006},
 zbMATH = {5614575},
 Zbl = {1183.20030}
}

@book{Wainwright2019,
 author = {Wainwright, M. J.},
 title = {High-dimensional statistics. {A} non-asymptotic viewpoint},
 fseries = {Cambridge Series in Statistical and Probabilistic Mathematics},
 series = {Camb. Ser. Stat. Probab. Math.},
 volume = {48},
 isbn = {978-1-108-49802-9; 978-1-108-62777-1},
 year = {2019},
 publisher = {Cambridge: Cambridge University Press},
 language = {English},
 doi = {10.1017/9781108627771},
 zbMATH = {7021501},
 Zbl = {1457.62011}
}

@book{VanDerVaart2000,
 author = {Van der Vaart, A. W.},
 title = {Asymptotic statistics},
 isbn = {0-521-78450-6},
 year = {2000},
 publisher = {Cambridge: Cambridge University Press},
 language = {English},
 doi = {10.1017/CBO9780511802256},
 zbMATH = {1495256},
 Zbl = {0943.62002}
}

@article{VenkovKoch1978,
 author = {Venkov, B. B. and Koch, H.},
 title = {The p-tower of class fields for an imaginary quadratic field},
 fjournal = {Journal of Soviet Mathematics},
 journal = {J. Sov. Math.},
 issn = {0090-4104},
 volume = {9},
 pages = {291--299},
 year = {1978},
 language = {English},
 doi = {10.1007/BF01085047},
 zbMATH = {3614892},
 Zbl = {0396.12010}
}

@misc{Luo2024,
 author = {Y. Luo},
 title = {On the {Boston}'s {Unramified} {Fontaine}-{Mazur} {Conjecture}},
 year = {2024},
 howpublished = {Preprint, {arXiv}:2404.18967 [math.{NT}] (2024)},
 url = {https://arxiv.org/abs/2404.18967},
 arXiv = {arXiv:2404.18967}
}

@article{KisinWortmann2003,
 author = {Kisin, M. and Wortmann, S.},
 title = {A note on {Artin} motives},
 fjournal = {Mathematical Research Letters},
 journal = {Math. Res. Lett.},
 issn = {1073-2780},
 volume = {10},
 number = {2-3},
 pages = {375--389},
 year = {2003},
 language = {English},
 doi = {10.4310/MRL.2003.v10.n3.a7},
 zbMATH = {2064165},
 Zbl = {1052.14022}
}

@article{Moonen2019,
 author = {Moonen, B.},
 title = {A remark on the {Tate} conjecture},
 fjournal = {Journal of Algebraic Geometry},
 journal = {J. Algebr. Geom.},
 issn = {1056-3911},
 volume = {28},
 number = {3},
 pages = {599--603},
 year = {2019},
 language = {English},
 doi = {10.1090/jag/720},
 zbMATH = {7061054},
 Zbl = {1448.11100}
}

@article{HajirMaire2002,
 author = {Hajir, F. and Maire, C.},
 title = {Unramified subextensions of ray class towers},
 fjournal = {Journal of Algebra},
 journal = {J. Algebra},
 issn = {0021-8693},
 volume = {249},
 number = {2},
 pages = {528--543},
 year = {2002},
 language = {English},
 doi = {10.1006/jabr.2001.9079},
 zbMATH = {1750367},
 Zbl = {1018.11058}
}

@article{GolodSafaravic1964,
 author = {Golod, E. S. and Shafarevich, I. R.},
 title = {On the class field tower},
 fjournal = {Izvestiya Akademii Nauk SSSR. Seriya Matematicheskaya},
 journal = {Izv. Akad. Nauk SSSR, Ser. Mat.},
 issn = {0373-2436},
 volume = {28},
 pages = {261--272},
 year = {1964},
 language = {Russian},
 zbMATH = {3220420},
 Zbl = {0136.02602}
}

@article{AhlqvistCarlson2025,
 author = {Ahlqvist, E. and Carlson, M.},
 title = {Massey products in the {\'e}tale cohomology of number fields},
 fjournal = {Journal f{\"u}r die Reine und Angewandte Mathematik},
 journal = {J. Reine Angew. Math.},
 issn = {0075-4102},
 volume = {823},
 pages = {61--112},
 year = {2025},
 language = {English},
 doi = {10.1515/crelle-2025-0006},
 zbMATH = {8049145}
}

@incollection{SymondsWeigel2000,
 author = {Symonds, P. and Weigel, T.},
 title = {Cohomology of {{\(p\)}}-adic analytic groups},
 booktitle = {New horizons in pro-\(p\) groups},
 isbn = {0-8176-4171-8},
 pages = {349--410},
 year = {2000},
 publisher = {Boston, MA: Birkh{\"a}user},
 language = {English},
 zbMATH = {1574590},
 Zbl = {0973.20043}
}

@article{LubotzkyMann1987I,
 author = {Lubotzky, A. and Mann, A.},
 title = {Powerful {{\(p\)}}-groups. {I}: {Finite} groups},
 fjournal = {Journal of Algebra},
 journal = {J. Algebra},
 issn = {0021-8693},
 volume = {105},
 pages = {484--505},
 year = {1987},
 language = {English},
 doi = {10.1016/0021-8693(87)90211-0},
 zbMATH = {4017157},
 Zbl = {0626.20010}
}

@article{LubotzkyMann1987II,
 author = {Lubotzky, A. and Mann, A.},
 title = {Powerful {{\(p\)}}-groups. {II}: {{\(p\)}}-adic analytic groups},
 fjournal = {Journal of Algebra},
 journal = {J. Algebra},
 issn = {0021-8693},
 volume = {105},
 pages = {506--515},
 year = {1987},
 language = {English},
 doi = {10.1016/0021-8693(87)90212-2},
 zbMATH = {4017169},
 Zbl = {0626.20022}
}

@article{HajirLarsenMaireRamakrishna2025,
 author = {Hajir, F. and Larsen, M. and Maire, C. and Ramakrishna, R.},
 title = {On tamely ramified infinite {Galois} extensions},
 fjournal = {Journal of the London Mathematical Society. Second Series},
 journal = {J. Lond. Math. Soc., II. Ser.},
 issn = {0024-6107},
 volume = {112},
 number = {1},
 pages = {27},
 note = {Id/No e70209},
 year = {2025},
 language = {English},
 doi = {10.1112/jlms.70209},
 zbMATH = {8070199}
}

@misc{OSCAR,
  key          = {OSCAR},
  autohor ={The {OSCAR} Team},
  organization = {The OSCAR Team},
  title        = {O{SCAR} -- {O}pen {S}ource {C}omputer {A}lgebra {R}esearch system, {V}ersion 1.7.0},
  year         = {2026},
  url          = {https://www.oscar-system.org},
}

@misc{Feuerpfeil2026Code,
author = {Feuerpfeil, J.},
title = {{Bockstein and Fontaine Mazur}},
year = {2026},
note = {\url{https://github.com/JulianFeuerpfeil/Bockstein_FontaineMazur}}
}

@article{AndozhskijTsvetkov1975,
 author = {Andozhskij, I. V. and Tsvetkov, V. M.},
 title = {On a series of finite closed p-groups},
 fjournal = {Mathematics of the USSR. Izvestiya},
 journal = {Math. USSR, Izv.},
 issn = {0025-5726},
 volume = {8},
 pages = {285--297},
 year = {1975},
 language = {English},
 doi = {10.1070/IM1974v008n02ABEH002105},
 zbMATH = {3534717},
 Zbl = {0343.20027}
}

@article{BrotoLevi1997,
 author = {Broto, C. and Levi, R.},
 title = {On the homotopy type of {{\(BG\)}} for certain finite 2-groups {{\(G\)}}},
 fjournal = {Transactions of the American Mathematical Society},
 journal = {Trans. Am. Math. Soc.},
 issn = {0002-9947},
 volume = {349},
 number = {4},
 pages = {1487--1502},
 year = {1997},
 language = {English},
 doi = {10.1090/S0002-9947-97-01692-9},
 zbMATH = {997379},
 Zbl = {0945.55012}
}

@article{Weigel2000,
 author = {Weigel, T.},
 title = {On the rigidity of {Lie} lattices and just infinite powerful groups},
 fjournal = {Journal of the London Mathematical Society. Second Series},
 journal = {J. Lond. Math. Soc., II. Ser.},
 issn = {0024-6107},
 volume = {62},
 number = {2},
 pages = {381--397},
 year = {2000},
 language = {English},
 doi = {10.1112/S0024610700001204},
 zbMATH = {1543013},
 Zbl = {1029.17020}
}

@article{DiazRuizViruel2013,
 author = {D{\'{\i}}az, A. and Ruiz, A. and Viruel, A.},
 title = {Cohomological uniqueness of some {{\(p\)}}-groups},
 fjournal = {Proceedings of the Edinburgh Mathematical Society. Series II},
 journal = {Proc. Edinb. Math. Soc., II. Ser.},
 issn = {0013-0915},
 volume = {56},
 number = {2},
 pages = {449--468},
 year = {2013},
 language = {English},
 doi = {10.1017/S0013091512000247},
 zbMATH = {6165744},
 Zbl = {1276.55019}
}
